\documentclass[11pt]{article} 
\usepackage{helvet}
\usepackage[a4paper, total={7in, 8.5in}]{geometry}

\usepackage{amsmath}
\usepackage{amsfonts}
\usepackage{amssymb}
\usepackage{natbib}
\usepackage[ruled,vlined]{algorithm2e}
\usepackage{epsfig}
\usepackage{bm}
\usepackage{makeidx}
\usepackage{natbib}
\usepackage{color}
\usepackage{amsthm}
\usepackage{framed}
\usepackage{float}
\usepackage{booktabs}
\usepackage{siunitx}
\usepackage{graphicx}
\usepackage{mathtools}
\usepackage{multirow, makecell}
\usepackage[normalem]{ulem}
\usepackage{geometry} % see geometry.pdf on how to lay out the page. There's lots.
\bibpunct{(}{)}{;}{a}{,}{,}
\usepackage{color}

\newcommand{\bn}{\begin{enumerate}}
\newcommand{\en}{\end{enumerate}}
\newcommand{\bi}{\begin{itemize}}
\newcommand{\ei}{\end{itemize}}
\newcommand{\be}{\begin{eqnarray}}
\newcommand{\ee}{\end{eqnarray}}
\newcommand{\by}{\begin{eqnarray*}}
\newcommand{\ey}{\end{eqnarray*}}
\newcommand{\beq}{\begin{equation}}
\newcommand{\eeq}{\end{equation}}

\newcommand{\defeq}{\vcentcolon=}

\newtheorem{theorem}{Theorem}[section]
\newtheorem{lemma}[theorem]{Lemma}

\newtheorem{corollary}[theorem]{Corollary}

\newcommand{\eE}{\mathbb{E}}
\newcommand{\KL}{\text{KL}}

\theoremstyle{plain}

\theoremstyle{plain}

\theoremstyle{remark}

\theoremstyle{plain}

 \title{\LARGE \textbf{Gibbs sampler and coordinate ascent variational inference: A set-theoretical review}}
\author{Se Yoon Lee
\\
\\
Department of Statistics, Texas A\&M University,  College Station, Texas, USA
\\
\\
\texttt{seyoonlee@stat.tamu.edu}
}
\date{}

\begin{document}
\maketitle
%\tableofcontents
\begin{abstract}
\noindent \\
One of the fundamental problems in Bayesian statistics is the approximation of the posterior distribution. Gibbs sampler and coordinate ascent variational inference are renownedly utilized approximation techniques that rely on stochastic and deterministic approximations. In this paper, we define fundamental sets of densities frequently used in Bayesian inference. We shall be concerned with the clarification of the two schemes from the set-theoretical point of view. This new way provides an alternative mechanism for analyzing the two schemes endowed with pedagogical insights. 
\end{abstract}

\section{Introduction}\label{sec:Introduction}
\noindent A statistical model contains a sample space of observations $\textbf{y}$ endowed with an appropriate $\sigma$-field of sets over which is given a family of probability measures. For almost all problems, it is sufficient to suppose that these probability measures can be described through their density functions, $p(\textbf{y}|\bm{\theta})$, indexed by a parameter $\bm{\theta}$ belonging to the parameter space $\bm{\Theta}$. In many problems, one of the essential goals is to make an inference about the parameter $\bm{\theta}$, and this article particularly concerns Bayesian inference. 

Bayesian approaches start with expressing the uncertainty associated with the parameter $\bm{\theta}$ through a density $\pi(\bm{\theta})$ supported on the parameter space $\bm{\Theta}$, called a prior. A collection $\{p(\textbf{y}|\bm{\theta}),\pi(\bm{\theta})\}$ is referred to as a Bayesian model. Given finite evidence $m(\textbf{y}) = \int p(\textbf{y}|\bm{\theta}) \cdot \pi(\bm{\theta}) d\bm{\theta}$ for all $\textbf{y}$, the Bayes' theorem formalizes an inversion process to learn the parameter $\bm{\theta}$ given the observations $\textbf{y}$ through its posterior distribution:
\begin{align}
\label{eq:posterior_distribution}
\pi(\bm{\theta}|\textbf{y}) = 
\frac{p(\textbf{y}|\bm{\theta}) \cdot \pi(\bm{\theta})}{m(\textbf{y})}.
\end{align}
A central task in the application of Bayesian models is the evaluation of this joint density $\pi(\bm{\theta}|\textbf{y})$ (\ref{eq:posterior_distribution}) or indeed to compute expectation with respect to this density. 

However, for many complex Bayesian models, the posterior distribution $\pi(\bm{\theta}|\textbf{y})$ is intractable. In such situations, we need to resort to approximation techniques, and these fall broadly into two classes, according to whether they rely on stochastic \citep{casella1992explaining,neal2011mcmc,murray2010elliptical,neal2003slice} or deterministic \citep{ranganath2014black,wang2013variational,minka2013expectation,blei2017variational} approximations. See \citep{andrieu2003introduction,zhang2018advances} for review papers for these techniques.

Gibbs sampler \citep{casella1992explaining} and coordinate ascent variational inference (CAVI) algorithm \citep{blei2017variational} are extremely popular techniques to approximate the target density $\pi(\bm{\theta}|\textbf{y})$ (\ref{eq:posterior_distribution}). They are often jointed with more sophisticated samplers or optimizers, and share some common structure from an implementational point of view. For instances, Gibbs sampler is combined with the Metropolis-Hastings algorithms \citep{chib1995understanding,beichl2000metropolis,dwivedi2018log} endowed with a nice proposal density which is typically easy to simulate from. The CAVI algorithm is combined with the stochastic gradient descent method \citep{ruder2016overview} endowed with a reasonable assumption of mean-field family whose members are computationally tractable.
% A History of the Metropolis–Hastings Algorithm: sentence structure

Essentially, the utilities of the two schemes are ascribed to their exploitations of the conditional independences \citep{dawid1979conditional} induced by a certain hierarchical structure formulated through the parameter $\bm{\theta}$ and observations $\textbf{y}$. That way, we can decompose the original problem of approximation of the joint density $\pi(\bm{\theta}|\textbf{y})$, possibly supported on a high-dimensional parameter space $\bm{\Theta}$, into a collection of small problems with low dimensionalities. A single cycle of resulting algorithms comprises multiple steps where at each step only a small fraction of the $\bm{\theta}$ is updated, while remaining components are fixed with the most recently updated information. 

This article aims to understand the two schemes set-theoretically to clarify some common structure between the two schemes and provide relevant pedagogical insights. Here, we say ``set-theoretical understanding" in the sense that we will treat fundamental densities used in the two schemes as elements of some sets of densities. Set-theoretical statements are helpful in a clear understanding of the algorithms as they explain how ingredients of the two schemes are functionally related each other. A duality formula for variational inference \citep{massart2007concentration} is the essential theorem for the expositions.

\section{A duality formula for variational inference}\label{sec:Duality formula for variational inference}
To state a duality formula for variational inference,  we first introduce some ingredients. Let $\bm{\Theta}$ be a set endowed with an appropriate $\sigma$-field $\mathcal{F}$, and two probability measures $P$ and $Q$, which formulates two probability spaces, $(\bm{\Theta}, \mathcal{F},P)$ and $(\bm{\Theta}, \mathcal{F},Q)$. We use notation $Q \ll P$ to indicate that $Q$ is absolutely continuous with respect to $P$ (i.e., $Q(A) = 0$ holds for any measurable set $A \in \mathcal{F}$ with $P(A) = 0$). Let $\eE_{P}[\cdot]$ denote integration with respect to the probability measure $P$. Given any real-valued random variable $g$ defined on the probability space $(\bm{\Theta}, \mathcal{F},P)$, notation $g \in L_{1}(P)$ implies that the random variable $g$ is integrable with respect to measure $P$, that is, $\eE_{P}[|g|]= \int |g| dP < \infty$. The notation $\KL(Q\|P)$ represents the Kullback-Leibler divergence from $P$ to $Q$, $\KL(Q\|P)=\int \log\ (dQ/dP) dQ$ 
 \citep{kullback1997information}.
 
\begin{theorem}[\textbf{Duality formula}]\label{theorem:Duality formula for variational inference}
Consider two probability spaces $(\bm{\Theta}, \mathcal{F},P)$ and $(\bm{\Theta}, \mathcal{F},Q)$ with $Q \ll P$. Assume that there is a common dominating probability measure $\lambda$ such that $P \ll \lambda$ and $Q\ll \lambda$. Let $h$ denote any real-valued random variable on $(\bm{\Theta}, \mathcal{F},P)$ that satisfies $\exp h \in L_{1}(P)$. Then the following equality holds
\begin{align*}
%\label{eq:duality_formula_original}
\log\eE_{P}[\exp h]
&= \text{sup}_{Q \ll P} \{
\eE_{Q}[h]
-
\KL(Q\|P) \}.
\end{align*}
Further, the supremum on the right-hand side is attained if and only if it holds
\begin{align*}
%\label{eq:optimum_duality_formula}
\frac{q(\theta)}{p(\theta)} = \frac{\exp h(\theta)}{\eE_{P}[\exp h]},
\end{align*}
almost surely with respect to probability measure $Q$, where $p(\theta) = dP/d\lambda$ and $q(\theta) = dQ/d\lambda$ denote the Radon-Nikodym derivatives of the probability measures $P$ and $Q$ with respect to $\lambda$, respectively.
\end{theorem}

\begin{proof}
We use a measurement theory \citep{royden1988real} for a direct proof. Due to the dominating assumptions $P \ll \lambda$ and $Q\ll \lambda$ and the Radon-Nikodym theorem (Theorem 32.1 of  \citep{billingsley2008probability}), there exist Radon-Nikodym derivatives (also called generalized probability densities \citep{kullback1997information}) 
$p(\theta) = dP/d\lambda$ and $q(\theta) = dQ/d\lambda$ unique up to sets of measure (probability) zero in $\lambda$ corresponding to measures $P$ and $Q$, respectively. On the other hand, due to the dominating assumption $Q \ll P$, there exists Radon-Nikodym derivative $dQ/dP$, hence, Kullback-Leibler divergence $\KL(Q\|P) = \int \log\ (dQ/dP) dQ$ is well-defined and finite. By using conventional measure-theoretic notation (for example, see page 4 of \citep{kullback1997information}), we shall write $dP(\theta) = p(\theta) d\lambda(\theta)$ and $dQ(\theta) = q(\theta) d\lambda(\theta)$, and $\int g dP =\int g(\theta) dP(\theta)$ for any $g \in L_{1}(P)$ \citep{resnick2003probability}.

Use the measure theoretic ingredients as follow:
\begin{align}
\nonumber
\eE_{Q}[h]
-
\KL(Q\|P)
&=
\int h dQ 
- 
\int \log\ \bigg(\frac{dQ}{dP} \bigg) dQ
\\
\nonumber
&=
\int h(\theta) dQ (\theta)
- 
\int \log\ \bigg(\frac{dQ(\theta)}{dP(\theta)} \bigg) dQ(\theta)\\
\nonumber
&=
\int h(\theta) q(\theta) d\lambda (\theta)
- 
\int \log\ \bigg(\frac{q(\theta)}{p(\theta)} \bigg) q(\theta) d\lambda (\theta)\\
\nonumber
&
=
\int \log\ \bigg(\frac{e^{h(\theta)}  p(\theta)}{q(\theta)} \bigg) q(\theta) d\lambda (\theta)\\
\label{eq:duality_proof}
&\leq
\log\bigg(
\int  \bigg(\frac{e^{h(\theta)}  p(\theta)}{q(\theta)} \bigg) q(\theta) d\lambda (\theta)
\bigg)
\\
\nonumber
&
=
\log
\bigg(
\int e^{h(\theta)}  p(\theta) d\lambda (\theta)
\bigg)
\\
\nonumber
&=
\log
\bigg(
\int e^{h(\theta)}  dP(\theta)
\bigg)
\\
\nonumber
&
=
\log
\bigg(
\int e^{h}  dP
\bigg)\\
\nonumber
&=
\log \eE_{P}[\exp h].
\end{align}
Note that the Jensen's inequality is used to derive the inequality in (\ref{eq:duality_proof}). Because the logarithm function is strictly concave, this inequality becomes the equality if and only if the function $e^{h(\theta)} p(\theta)/q(\theta)$ is constant on $\bm{\Theta}$ almost surely with respect to measure $Q$ (page 52 of \citep{keener2010theoretical}). Let $e^{h(\theta)} p(\theta)/q(\theta) = k$, where $k$ is constant almost surely with respect to $Q$. Then we have $e^{h(\theta)} p(\theta)= k q(\theta)$. Take the integral $\int \cdot d\lambda(\theta)$ on both sides on the final equation to complete the proof: $\mathbb{E}_{P}[\exp h]= \int e^{h} dP = \int e^{h(\theta)} dP(\theta)=
  \int e^{h(\theta)} p(\theta)d\lambda(\theta) = \int k q(\theta)d\lambda(\theta) =\int k dQ(\theta)=  \int  k dQ = k$.
\end{proof}

In practice, a common dominating measure $\lambda$ for $P$ and $Q$ is usually either Lebesgue or counting measure. This article mainly explores the former case where the duality formula becomes
\begin{align}
\label{eq:duality_formula_continuous_version}
\log\eE_{p(\theta)}[\exp h(\theta)]
&= \text{sup}_{q \ll p} \{
\eE_{q(\theta)}[h(\theta)]
-
\KL(q\|p) \}.
\end{align}
In equality (\ref{eq:duality_formula_continuous_version}), the $p(\theta)=dP/d\lambda$ and $q(\theta)=dQ/d\lambda$ are probability density functions (pdf) corresponding to the probability measures $P$ and $Q$, respectively, and $h(\theta)$ is any measurable function such that the expectation $\eE_{p(\theta)}[\exp h(\theta)]$ is finite. Expectations in the equilibrium (\ref{eq:duality_formula_continuous_version}) are taken with respect to densities on the subscripts. For instance, the expectation $\eE_{p(\theta)}[\exp h(\theta)]$ represents the integral $\int \exp h(\theta) p(\theta) d\theta$, and the Kullback-Leibler divergence is expressed with the pdf version, that is, $\KL(q(\theta)\|p(\theta)) = \int q(\theta) \log (q(\theta)/p(\theta)) d\theta$. We shall use the notation $q \ll p$ to indicate that their corresponding probability measures satisfy the dominating condition $Q \ll P$.

\section{Fundamental sets in Bayesian inference}\label{sec:Essential sets in Bayesian reasoning}
Consider a Bayesian model $\{p(\textbf{y}|\bm{\theta}),\pi(\bm{\theta})\}$ where $p(\textbf{y}|\bm{\theta})$ is a data generating process and $\pi(\bm{\theta})$ is a prior density as explained in Section \ref{sec:Introduction}. Now, we additionally assume that the entire parameter space $\bm{\Theta}$ is decomposed as
\begin{align}
\label{eq:decomposition_of_entire_ps}
\bm{\Theta} = \Pi_{i=1}^K \Theta_{i}=  \Theta_{1} \times  \cdots   \times \Theta_{i} \times \cdots  \times \Theta_{K},
\end{align}
for some integer $K>1$. Each component parameter space $\Theta_{i}$ ($i=1,\cdots,K$) can be a set of scalar components, subvectors, or matrices \citep{smith1993bayesian}. The notation $A \times B$ denotes the Caresian product between two sets $A$ and $B$. Under the decomposition (\ref{eq:decomposition_of_entire_ps}), elements of the set $\bm{\Theta}$ can be expressed as $\bm{\theta}=(\theta_{1}, \cdots, \theta_{i} , \cdots, \theta_{K}) \in \bm{\Theta}$ where $\theta_{i}\in \Theta_{i}$ ($i=1,\cdots,K$).

For each $i$, we define a set that complements the $i$-th component parameter space $\Theta_{i}$:
\begin{align}
\label{eq:complementary_ps}
\Theta_{-i} \defeq \Pi_{j =1, j \neq i }^K \Theta_{j}
 =  \Theta_{1} \times  \cdots   \times \Theta_{i-1} \times 
\Theta_{i+1}  \times
 \cdots  \times \Theta_{K}.
\end{align}
The set $\Theta_{-i}$ (\ref{eq:complementary_ps}) is called the $i$-th complementary parameter space. Elements of the $\Theta_{-i}$ are of the form $\theta_{-i}=(\theta_{1}, \cdots, \theta_{i-1},  \theta_{i+1}, \cdots, \theta_{K}) \in\Theta_{-i}$. 

\begin{figure}[ht]
\centering
\includegraphics[width=1\textwidth]{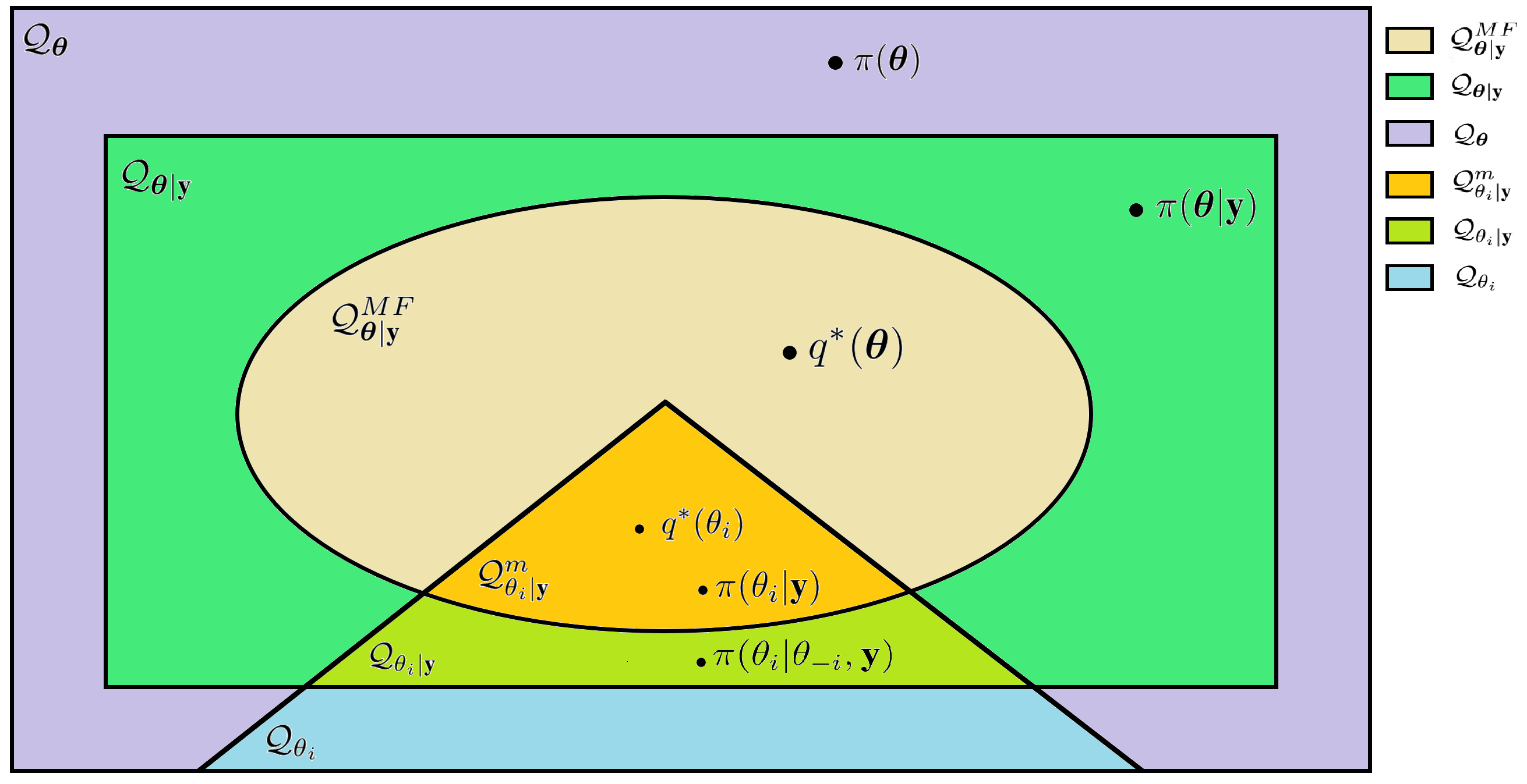}
\caption{Venn diagram that superimposed two set-inclusion relationships: (1) $\mathcal{Q}_{\bm{\theta}|\textbf{y}}^{MF} \subset \mathcal{Q}_{\bm{\theta}|\textbf{y}} \subset \mathcal{Q}_{\bm{\theta}}$, and (2) $\mathcal{Q}_{\theta_{i}|\textbf{y}}^{m} \subset \mathcal{Q}_{\theta_{i}|\textbf{y}} \subset \mathcal{Q}_{\theta_{i}}$ for each component index $i$ ($i=1,\cdots, K$). Symbol $\bullet$ indicates an element of the sets. For each $i$, $\pi(\theta_{i}|\theta_{-i},\textbf{y})$ (\ref{eq:full_conditional_for_theta_i}), $\pi(\theta_{i}|\textbf{y})$, and $q^{*}(\theta_{i})$ (\ref{eq:variational_factor_theta_i}) are the full conditional, marginal posterior, and variational factor, respectively. The $\pi(\bm{\theta})$, $\pi(\bm{\theta}|\textbf{y})$ (\ref{eq:posterior_distribution}), and $q^{*}(\bm{\theta})$ are prior density, posterior density, and VB posterior, respectively.}
\label{fig:set_inclusion}
\end{figure}

It is important to emphasize that how to impose a decomposition on the set $\bm{\Theta}$ (i.e., to determine the integer $K$ or the dimension of the component parameter spaces $\Theta_{i}$ in (\ref{eq:decomposition_of_entire_ps})) is at the discretion of a model builder. Although it is possible to fully factorize the $\bm{\Theta}$ so that every $\Theta_{i}$ consists of scalar components, sampling or optimization algorithm based on such a decomposition may have limited accuracy, especially when the latent variables are highly dependent. One important consideration in choosing a decomposition (\ref{eq:decomposition_of_entire_ps}) is the correlation structure embedded in the target density $\pi(\bm{\theta}|\textbf{y})$. For instance, when a Bayesian model $\{p(\textbf{y}|\bm{\theta}),\pi(\bm{\theta})\}$ formulates a certain hierarchical structure, one can impose a decomposition based on the conditional independence induced by the hierarchy among the latent variables $\theta_{i}$'s and observations $\textbf{y}$, thereby exploiting the notion of Markov blankets \citep{pellet2008using}.

Now, we define fundamental sets of densities, itemized with \textbf{(i)} -- \textbf{(vi)}. They play crucial roles in Bayesian inference about the parameter $\bm{\theta}$ provided that Gibbs sampler or CAVI algorithm is carried out to approximate the target density $\pi(\bm{\theta}|\textbf{y})$ (\ref{eq:posterior_distribution}):
\begin{itemize}
\item[\textbf{(i)}]Set $\mathcal{Q}_{\bm{\theta}}$ is the collection of all the densities supported on the parameter space $\bm{\Theta}$. Thus, the $\mathcal{Q}_{\bm{\theta}}$ is the largest set of densities that we can consider in Bayesian inference. Set $\mathcal{Q}_{\bm{\theta}|\textbf{y}}$ is the collection of all the densities conditioned on the observations $\textbf{y}$. Here, the term `conditioned on the observations $\textbf{y}$' can be also replaced by term `having observed $\textbf{y}$', or more concisely by `\emph{a posteriori}'. By definitions, it holds a subset inclusion, $\mathcal{Q}_{\bm{\theta}|\textbf{y}} \subset \mathcal{Q}_{\bm{\theta}}$. Both prior and posterior densities, $\pi(\bm{\theta})$ and $\pi(\bm{\theta}|\textbf{y})$ (\ref{eq:posterior_distribution}), belong to the set $\mathcal{Q}_{\bm{\theta}}$. And in particular, the posterior density $\pi(\bm{\theta}|\textbf{y})$ belongs to the set $\mathcal{Q}_{\bm{\theta}|\textbf{y}}$. As the prior density $\pi(\bm{\theta})$ does not involve the observations $\textbf{y}$, it holds $\pi(\bm{\theta})\in \mathcal{Q}_{\bm{\theta}} - \mathcal{Q}_{\bm{\theta}|\textbf{y}} = \mathcal{Q}_{\bm{\theta}} \cap (\mathcal{Q}_{\bm{\theta}|\textbf{y}})^{c}$;
\item[\textbf{(ii)}] For each $i$ ($i=1,\cdots,K$), set $\mathcal{Q}_{\theta_{i}}$ is the collection of all the densities supported on the $i$-th component parameter space $\Theta_{i}$, and set $\mathcal{Q}_{\theta_{i}|\textbf{y}}$ denotes the collection for the only posterior densities supported on $\Theta_{i}$. This implies that a subset inclusion, $\mathcal{Q}_{\theta_{i}|\textbf{y}} \subset \mathcal{Q}_{\theta_{i}}$ holds for each $i$. For each $i$, the full conditional posterior density $\pi(\theta_{i}|\theta_{-i},\textbf{y}) = \pi(\theta_{i},\theta_{-i},\textbf{y})/\pi(\theta_{-i},\textbf{y}) = \pi(\bm{\theta},\textbf{y})/\pi(\theta_{-i},\textbf{y})$ and marginal posterior density $\pi(\theta_{i}|\textbf{y})$ are typical elements of the set $\mathcal{Q}_{\theta_{i}|\textbf{y}}$;
\item[\textbf{(iii)}] For each $i$ ($i=1,\cdots,K$), set $\mathcal{Q}_{\theta_{-i}}$ is the collection of all the densities supported on the $i$-th complementary parameter space $\Theta_{-i}$ (\ref{eq:complementary_ps}), and set $\mathcal{Q}_{\theta_{-i}|\textbf{y}}$ denotes the collection for the only posterior densities supported on $\Theta_{-i}$. By definition, it holds $\mathcal{Q}_{\theta_{-i}|\textbf{y}} \subset \mathcal{Q}_{\theta_{-i}}$ for each $i$. Two typical elements of the set $\mathcal{Q}_{\theta_{-i}|\textbf{y}}$ are $\pi(\theta_{-i}|\textbf{y})$ which satisfies $\pi(\bm{\theta}|\textbf{y}) = \pi(\theta_{i}|\theta_{-i},\textbf{y}) \pi(\theta_{-i}|\textbf{y})$, and $\pi(\theta_{-i}|\theta_{i}, \textbf{y})$ which satisfies $\pi(\bm{\theta}|\textbf{y}) = \pi(\theta_{-i}|\theta_{i},\textbf{y}) \pi(\theta_{i}|\textbf{y})$;
\item[\textbf{(iv)}] For each $i$ ($i=1,\cdots,K$), set $\mathcal{Q}_{\theta_{i}}^{m}$ is the collection of all the `marginal' densities supported on the $i$-th component parameter space $\Theta_{i}$, and set $\mathcal{Q}_{\theta_{i}|\textbf{y}}^{m}$ is the collection for the only `posterior marginal' densities supported on $\Theta_{i}$. Here, the `marginal' is superscripted with `$m$'. The meaning of marginal density of $\theta_{i}$ can be understood that no elements in the $i$-th complementary parameter space $\Theta_{-i}$ (\ref{eq:complementary_ps}) are involved in the density. Although the marginal posterior density $\pi(\theta_{i}|\textbf{y})\in \mathcal{Q}_{\theta_{i}|\textbf{y}}$ belongs to the set $\mathcal{Q}_{\theta_{i}|\textbf{y}}^{m}$, the full conditional density $\pi(\theta_{i}|\theta_{-i},\textbf{y}) \in \mathcal{Q}_{\theta_{i}|\textbf{y}}$ does not belong to the set $\mathcal{Q}_{\theta_{i}|\textbf{y}}^{m}$ unless $\theta_{i}$ and $\theta_{-i}$ are conditionally independent given $\textbf{y}$; 
%A typical Bayesian way to make a non-marginal density of $\theta_{i}$ to a marginal one is to integrate the former over $\Theta_{-i}$ (\ref{eq:complementary_ps}). For example, given full conditional $\pi(\theta_{i}|\theta_{-i},\textbf{y})$, we can employ the integration $\int_{\Theta_{-i}} \pi(\theta_{i}|\theta_{-i},\textbf{y}) q(\theta_{-i}) d\theta_{-i}$ for any density $q(\theta_{-i})\in \mathcal{Q}_{\theta_{-i}}$ to produce a marginal density of $\theta_{i}$. 
\item[\textbf{(v)}] Cartesian product of $K$ sets of marginal densities $\{\mathcal{Q}_{\theta_{i}}^{m}\}_{i=1}^{K}$ (defined in the item \textbf{(iv)}) produces a set of densities supported on $\bm{\Theta}$
\begin{align}
\label{eq:mean-field variational family}
\mathcal{Q}_{\bm{\theta}}^{MF} & \defeq
\prod_{i=1}^{K} \mathcal{Q}_{\theta_{i}}^{m} =
\mathcal{Q}_{\theta_{1}}^{m}\times \cdots \times 
\mathcal{Q}_{\theta_{i}}^{m}
\times
\cdots
\mathcal{Q}_{\theta_{K}}^{m}=
\bigg\{q(\bm{\theta})
\,\bigg|\,q(\bm{\theta}) = \prod_{i=1}^{K} q(\theta_{i}),q(\theta_{i}) \in \mathcal{Q}_{\theta_{i}}^{m} 
\bigg\}
.
\end{align}
The set $\mathcal{Q}_{\bm{\theta}}^{MF}$ (\ref{eq:mean-field variational family}) is referred to as the \emph{mean-field variational family} \citep{jordan1999introduction}, whose root can be found in statistical physics literature \citep{chandler1987introduction,parisi1988statistical,baxter2016exactly}. The superscript `$MF$' represents the `mean-field'. 

Note that an element of the set $\mathcal{Q}_{\bm{\theta}}^{MF}$ (\ref{eq:mean-field variational family}) is expressed with a product-form distribution supported on the parameter space $\bm{\Theta}$ (\ref{eq:decomposition_of_entire_ps}). Due to the definition of the set $\mathcal{Q}_{\theta_{i}}^{m}$ ($i=1,\cdots,K$) (defined in \textbf{(iv)}), elements of the set $\mathcal{Q}_{\bm{\theta}}^{MF}$ (\ref{eq:mean-field variational family}) retain a flexibility, a nice feature of non-parametric densities, with the unique constraint on the flexibility is the (marginal) independence among $\theta_{i}$'s induced by the mean-field theory (\ref{eq:mean-field variational family}) \citep{ormerod2010explaining}. It is important to emphasize that this mean-field assumption (\ref{eq:mean-field variational family}) is \emph{not} a modeling assumption underlying the Bayesian model $\{p(\textbf{y}|\bm{\theta}),\pi(\bm{\theta})\}$: that is, we do not need this assumption to approximate the target density $\pi(\bm{\theta}|\textbf{y})$ (\ref{eq:posterior_distribution}). Rather, the purpose of imposing the mean-field assumption is to implement the CAVI algorihtm.

Likewisely, we define a set $\mathcal{Q}_{\bm{\theta}|\textbf{y}}^{MF}$ as the set obtained by Cartesian product of $K$ sets of posterior marginal densities 
$\{\mathcal{Q}_{\theta_{i}|\textbf{y}}^{m}\}_{i=1}^{K}$ as follows
\begin{align}
\label{eq:mean-field variational family_a_posteriori}
\mathcal{Q}_{\bm{\theta}|\textbf{y}}^{MF} \defeq \prod_{i=1}^{K}\mathcal{Q}_{\theta_{i}|\textbf{y}}^{m}=
\mathcal{Q}_{\theta_{1}|\textbf{y}}^{m}\times \cdots 
\times 
\mathcal{Q}_{\theta_{i}|\textbf{y}}^{m}\times 
\cdots 
\times \mathcal{Q}_{\theta_{K}|\textbf{y}}^{m};
\end{align}
\item[\textbf{(vi)}] For each $i$ ($i=1,\cdots, K$), Cartesian product of $K-1$ sets of marginal densities $\{\mathcal{Q}_{\theta_{j}}^{m}\}_{j=1,j\neq i}^{K}$ (defined in \textbf{(iv)}) defines a set of densities supported on the $i$-th complementary parameter space $\Theta_{-i}$ (\ref{eq:complementary_ps})
\begin{align}
\label{eq:the_i_th_complemental_variational_family}
&\mathcal{Q}_{\theta_{-i}}^{MF} \defeq 
\prod_{j=1, j\neq i}^{K} \mathcal{Q}_{\theta_{j}}^{m}=
\mathcal{Q}_{\theta_{1}}^{m}\times  \cdots \times \mathcal{Q}_{\theta_{i-1}}^{m}\times  \mathcal{Q}_{\theta_{i+1}}^{m} \times  \cdots \times  \mathcal{Q}_{\theta_{K}}^{m}\\
\nonumber
&\,\,=
\bigg\{q(\theta_{-i}) 
\,\bigg|\,q(\theta_{-i}) 
=
\prod_{j=1,j\neq i}^{K}q(\theta_{j})
= q(\theta_{1})\cdots q(\theta_{i-1})
\cdot q(\theta_{i+1}) \cdots q(\theta_{K}), \,q(\theta_{j}) \in \mathcal{Q}_{\theta_{j}}^{m} 
\bigg\}.
\end{align}

Similarly, we define a set $\mathcal{Q}_{\theta_{-i}|\textbf{y}}^{MF}$ as the set obtained by Cartesian product of $K-1$ sets of posterior marginal densities $\{\mathcal{Q}_{\theta_{j}|\textbf{y}}^{m}\}_{j=1,j\neq i}^{K}$
\begin{align}
\label{eq:the_i_th_complemental_variational_family_posterirori}
\mathcal{Q}_{\theta_{-i}|\textbf{y}}^{MF}  
&\defeq 
\prod_{j=1, j\neq i}^{K} \mathcal{Q}_{\theta_{j}|\textbf{y}}^{m}=
\mathcal{Q}_{\theta_{1}|\textbf{y}}^{m}\times  \cdots \times \mathcal{Q}_{\theta_{i-1}|\textbf{y}}^{m} \times  \mathcal{Q}_{\theta_{i+1}|\textbf{y}}^{m} \times  \cdots \times  \mathcal{Q}_{\theta_{K}|\textbf{y}}^{m}.
\end{align}
\end{itemize}

Figure \ref{fig:set_inclusion} shows a Venn diagram which depicts set-inclusion relationships formed from the fundamental sets defined in items \textbf{(i)} $-$ \textbf{(vi)}. Some key elements (hence, densities) are marked by symbol $\bullet$. As seen from the panel, by notational definition, two chains of subset-inclusion hold: (1) densities supported on the entire parameter space $\bm{\Theta}$, $\mathcal{Q}_{\bm{\theta}|\textbf{y}}^{MF} \subset \mathcal{Q}_{\bm{\theta}|\textbf{y}} \subset \mathcal{Q}_{\bm{\theta}}$; and (2) densities supported on the $i$-th component parameter space $\Theta_{i}$, $\mathcal{Q}_{\theta_{i}|\textbf{y}}^{m} \subset \mathcal{Q}_{\theta_{i}|\textbf{y}} \subset \mathcal{Q}_{\theta_{i}}$, for each $i=1,\cdots,K$. Here, it holds $\mathcal{Q}_{\bm{\theta}|\textbf{y}}^{MF} \subset \mathcal{Q}_{\bm{\theta}|\textbf{y}}$ because, by definitions, the $\mathcal{Q}_{\bm{\theta}|\textbf{y}}$ is the set of `all the densities' supported on the $\bm{\Theta}$, having observed the $\textbf{y}$, whereas the $\mathcal{Q}_{\bm{\theta}|\textbf{y}}^{MF}$ (\ref{eq:mean-field variational family_a_posteriori}) is the set of `all the densities of the product-form' supported on the $\bm{\Theta}$, having observed the $\textbf{y}$. Furthermore, it holds a proper subset inclusion relationship $\mathcal{Q}_{\bm{\theta}|\textbf{y}}^{MF} \subsetneq \mathcal{Q}_{\bm{\theta}|\textbf{y}}$ as the integer $K$ is greater than $1$.

Because the Venn diagram overlaid these chains on a single panel for visualization purpose, it should not be interpreted that subset-inclusions $\mathcal{Q}_{\theta_{i}|\textbf{y}}^{m} \subset \mathcal{Q}_{\bm{\theta}|\textbf{y}}^{MF}$, $\mathcal{Q}_{\theta_{i}|\textbf{y}} \subset \mathcal{Q}_{\bm{\theta}|\textbf{y}}$, and $\mathcal{Q}_{\theta_{i}} \subset \mathcal{Q}_{\bm{\theta}}$ hold for each $i$ ($i=1,\cdots,K$). Rather, it should be interpreted that each of the sets $\mathcal{Q}_{\theta_{i}|\textbf{y}}^{m}$, $\mathcal{Q}_{\theta_{i}|\textbf{y}}$, and $\mathcal{Q}_{\theta_{i}}$ participates to each of the sets $\mathcal{Q}_{\bm{\theta}|\textbf{y}}^{MF}$, $\mathcal{Q}_{\bm{\theta}|\textbf{y}}$, and $\mathcal{Q}_{\bm{\theta}}$ as a piece, respectively. 

\section{Gibbs sampler and CAVI algorithm}\label{sec:A review for a Gibbs sampler and CAVI algorithm}
\subsection{\textbf{Gibbs sampling algorithm}}\label{subsec:Gibbs sampling algorithm}
Consider a Bayesian model $\{p(\textbf{y}|\bm{\theta}),\pi(\bm{\theta})\}$ as illustrated in Section \ref{sec:Introduction}. The Gibbs sampler algorithm \citep{geman1984stochastic,casella1992explaining} is a Markov chain Monte Carlo (MCMC) sampling scheme to approximate the target density $\pi(\bm{\theta}|\textbf{y}) \in \mathcal{Q}_{\bm{\theta}|\textbf{y}}$ (\ref{eq:posterior_distribution}). A single cycle of the Gibbs sampler is executed by iteratively drawing a sample from each of the full conditional posterior densities
\begin{align}
\label{eq:full_conditional_for_theta_i}
\pi(\theta_{i}|\theta_{-i},\textbf{y}) 
=
\pi(\theta_{i}|\theta_{1},\cdots,\theta_{i-1},\theta_{i+1},\cdots,\theta_{K}, \textbf{y}) 
\in \mathcal{Q}_{\theta_{i}|\textbf{y}}, \quad (i=1,\cdots, K),
\end{align}
while fixing other full conditional posterior densities. In each of the $K$ steps within a cycle, latent variables conditioned on the density (\ref{eq:full_conditional_for_theta_i}) (i.e., $\theta_{-i} \in \Theta_{-i} $) are updated with the most recently drawn values. See  \citep{gelfand1990sampling,gelfand2000gibbs} for a comprehensive review of Gibbs sampler. Algorithm \ref{alg:Gibbs sampler} details a generic Gibbs sampler.

% http://www.mit.edu/~ilkery/papers/GibbsSampling.pdf
\begin{algorithm}[h]
\caption{Gibbs sampler}\label{alg:Gibbs sampler}
\SetAlgoLined
\textbf{Initialize: } pick arbitrary starting value $\bm{\theta}^{(1)} = (\theta_{1}^{(1)},\theta_{2}^{(1)},\cdots, \theta_{K}^{(1)} ) \sim h(\bm{\theta}) \in \mathcal{Q}_{\bm{\theta}}$\\
\textbf{Iterate a cycle:}\\
$\quad$ $\emph{Step 1.}$ $\text{draw}$  $\theta_{1}^{(s+1)} \sim \pi(\theta_{1}|\theta_{2}^{(s)},\theta_{3}^{(s)},\cdots,\theta_{K}^{(s)},\textbf{y}) \in \mathcal{Q}_{\theta_{1}|\textbf{y}}$\\
$\quad$ $\emph{Step 2.}$ $\text{draw}$ $\theta_{2}^{(s+1)} \sim \pi(\theta_{2}|\theta_{1}^{(s+1)},\theta_{3}^{(s)},\cdots,\theta_{K}^{(s)},\textbf{y})\in \mathcal{Q}_{\theta_{2}|\textbf{y}}$\\
$\quad\quad\quad\quad\quad\quad\quad\quad\quad$ $\vdots$\\
$\quad$ $\emph{Step K.}$ $\text{draw}$ $\theta_{K}^{(s+1)} \sim \pi(\theta_{K}|\theta_{1}^{(s+1)},\theta_{2}^{(s+1)},\cdots,\theta_{K-1}^{(s+1)},\textbf{y})\in \mathcal{Q}_{\theta_{K}|\textbf{y}}$\\
\textbf{end Iterate}
\end{algorithm}
%https://www.math.arizona.edu/~tgk/mc/book_chap8.pdf
Algorithm \ref{alg:Gibbs sampler} produces a Markov chain $\{\bm{\theta}^{(1)}$, $\cdots$, $\bm{\theta}^{(s)}$, $\bm{\theta}^{(s+1)}$, $\cdots\}$ on the parameter space $\bm{\Theta}$. The transition kernel $K_{\text{G}}(\cdot,\cdot): \bm{\Theta} \times \bm{\Theta} \rightarrow [0,\infty)$ from $\bm{\theta}^{(s)} = (\theta_{1}^{(s)},\theta_{2}^{(s)},\cdots,\theta_{K}^{(s)})$ to $\bm{\theta}^{(s+1)} = (\theta_{1}^{(s+1)},\theta_{2}^{(s+1)},\cdots,\theta_{K}^{(s+1)})$ underlying the Gibbs sampler is
\begin{align}
\label{eq:transition_kernel_Gibbs}
K_{\text{G}}(\bm{\theta}^{(s)},\bm{\theta}^{(s+1)})
&=
\prod_{i=1}^{K}
\pi(\theta_{i}^{(s+1)}|\theta_{j}^{(s)}, j > i ,
\theta_{j}^{(s+1)}, j < i ,\textbf{y}).
\end{align}
Considerable theoretical works have been done on establishing the convergence of the Gibbs sampler for particular applications \citep{roberts1994geometric,roberts1994simple,cowles1996markov}. Under a mild condition (for example, Lemma 1 in \citep{smith1993bayesian}), one can prove that the stationary distribution (or the invariant distribution) for the above Markov chain is the target density $\pi(\bm{\theta}|\textbf{y})$ (\ref{eq:posterior_distribution}).
\subsection{\textbf{Coordinate ascent variational inference algorithm}}\label{subsec:Coordinate ascent variational inference algorithm}
Consider a Bayesian model $\{p(\textbf{y}|\bm{\theta}),\pi(\bm{\theta})\}$ as explained in Section \ref{sec:Introduction}. Variational inference is a deterministic functional optimization method to approximate the target density $\pi(\bm{\theta}|\textbf{y})\in \mathcal{Q}_{\bm{\theta}|\textbf{y}}$ (\ref{eq:posterior_distribution}) with another density $q(\bm{\theta}) \in \widetilde{\mathcal{Q}}_{\bm{\theta}}  \subseteq \mathcal{Q}_{\bm{\theta}}$ where $\widetilde{\mathcal{Q}}_{\bm{\theta}}$ is a set of candidate densities. To induce a good approximation, we wish that the set $\widetilde{\mathcal{Q}}_{\bm{\theta}}$ contains some nice elements close enough to the target density $\pi(\bm{\theta}|\textbf{y})$. It is important to emphasize that the Gibbs sampler (or any other MCMC/MC sampling techniques) presumes that it holds $\widetilde{\mathcal{Q}}_{\bm{\theta}} = \mathcal{Q}_{\bm{\theta}}$. In contrast, when implementing variational inference techniques, we often specify the set $\widetilde{\mathcal{Q}}_{\bm{\theta}} $ as a proper subset of $\mathcal{Q}_{\bm{\theta}}$ (i.e., $\widetilde{\mathcal{Q}}_{\bm{\theta}} \subsetneq \mathcal{Q}_{\bm{\theta}}$), yet the set $\widetilde{\mathcal{Q}}_{\bm{\theta}}$ still contains some nice candidate densities which can be computationally feasible to find.

Mean-field variational inference (MFVI) \citep{beal2003variational,jordan1999introduction} is a special kind of variational inference, principled on mean-field theory \citep{chandler1987introduction}, where candidate densities are searched over the product-form distributions $q(\bm{\theta})\in \widetilde{\mathcal{Q}}_{\bm{\theta}} = \mathcal{Q}_{\bm{\theta}}^{MF}\subsetneq \mathcal{Q}_{\bm{\theta}}$ (\ref{eq:mean-field variational family}). The theoretical aim of MFVI is to minimize the Kullback-Leibler divergence between $q(\bm{\theta})$ and $\pi(\bm{\theta}|\textbf{y})$:
\begin{align}
\label{eq:variational_density_general}
\hat{q}^{MF}(\bm{\theta}) &\defeq  
\text{argmin}_{q(\bm{\theta}) \in \mathcal{Q}_{\bm{\theta}}^{MF}}
\text{KL}(q(\bm{\theta})||\pi(\bm{\theta}|\textbf{y}))\in \mathcal{Q}_{\bm{\theta}|\textbf{y}}^{MF}.
\end{align}
The output density $\hat{q}^{MF}(\bm{\theta})$ is referred to as the global minimizer \citep{Zhang_AOS}. In practice, the global minimizer is not very useful. This is because for complex Bayesian machine learning models such as Bayesian deep learning \citep{gal2016uncertainty}, latent Dirichlet allocation \citep{blei2003latent}, etc, the dimension of the parameter space $\bm{\Theta}$ is very high, and finding the global minimizer $\hat{q}^{MF}(\bm{\theta})$ directly from the set $\mathcal{Q}_{\bm{\theta}}^{MF}$ is often computationally infeasible.

The CAVI algorithm \citep{bishop2006pattern,blei2017variational} is an algorithm addressing this computational issue by iteratively minimizing the Kullback-Leibler divergence between $q(\theta_{i}) \in \mathcal{Q}_{\theta_{i}}^{m}$ and $\pi(\bm{\theta}|\textbf{y})$ while fixing others $q^{*}(\theta_{j})$ ($j\neq i$) to be most recently updated ones:
\begin{align}
\label{eq:CAVI_KL_expression}
q^{*}(\theta_{i})&\leftarrow
\text{argmin}_{
q(\theta_{i}) \in \mathcal{Q}_{\theta_{i}}^{m}
}
\text{KL}\bigg(
q(\theta_{i})
\cdot
\prod_{j=1,j\neq i}^{K}q^{*}(\theta_{j})
\bigg|\bigg|\pi(\bm{\theta}|\textbf{y})
\bigg)
\in \mathcal{Q}_{\theta_{i}|\textbf{y}}^{m}.
\end{align}
In (\ref{eq:CAVI_KL_expression}), superscript $*$ is used to indicate that the corresponding density has been \emph{updated}. The output $q^{*}(\theta_{i})$ ($i=1,\cdots,K$) (\ref{eq:CAVI_KL_expression}) is referred to as the $i$-th variational factor \citep{blei2017variational}. The final output of the CAVI algorithm is a product-form joint density $q^{*}(\bm{\theta}) = \prod_{i=1}^{K}q^{*}(\theta_{i}) \in \mathcal{Q}_{\bm{\theta}|\textbf{y}}^{MF}$ (\ref{eq:mean-field variational family_a_posteriori}) obtained by repeating the iteration (\ref{eq:CAVI_KL_expression}) until the Kullback-Leibler divergence on the right-hand side of (\ref{eq:CAVI_KL_expression}) is small enough: this final output $q^{*}(\bm{\theta})$ is referred to as the VB posterior \citep{wang2019frequentist}. We can obtain the closed-form expression of the variational factor $q^{*}(\theta_{i})$ (\ref{eq:CAVI_KL_expression}):
\begin{lemma} 
Consider a Bayesian model $\{p(\textbf{y}|\bm{\theta})$,$\pi(\bm{\theta})\}$ with the entire parameter space $\bm{\Theta}$ decomposed as (\ref{eq:decomposition_of_entire_ps}). Provided mean-field variational family $\mathcal{Q}_{\bm{\theta}}^{MF}$ (\ref{eq:mean-field variational family}), assume that the CAVI algorithm (\ref{eq:CAVI_KL_expression}) is employed to approximate the target density $\pi(\bm{\theta}|\textbf{y})$ (\ref{eq:posterior_distribution}).
\\
\noindent Then the $i$-th variational factor $q^{*}(\theta_{i})$ is
\begin{align}
\label{eq:variational_factor_theta_i}
q^{*}(\theta_{i}) &= \frac{\exp\ \mathbb{E}_{q^{*}(\theta_{-i})}[  \log\ \pi(\theta_{i}|\theta_{-i},\textbf{y}) ]}{\int \exp\ \mathbb{E}_{q^{*}(\theta_{-i})}[  \log\ \pi(\theta_{i}|\theta_{-i},\textbf{y}) ] d\theta_{i}},
\end{align}
where $q^{*}(\theta_{-i}) = \prod_{j=1,j\neq i}^{K}q^{*}(\theta_{j})$. 
\end{lemma}
\begin{proof}
To avoid notation clutters, we omit the superscipt $*$ on the $q^{*}(\theta_{i})$ and $q^{*}(\theta_{-i})$. We start with splitting the divergence $\text{KL}(q(\theta_{i})\cdot q(\theta_{-i})||
\pi(\bm{\theta}|\textbf{y}))$ into two integrals
\begin{align}
\nonumber
&\text{KL}(q(\theta_{i})\cdot q(\theta_{-i})||
\pi(\bm{\theta}|\textbf{y}))\\
\label{eq:variational_factor_proof_1}
&\,\,=
\int
q(\theta_{i})\cdot q(\theta_{-i})
\log \{q(\theta_{i}) \cdot q(\theta_{-i})\}
d\bm{\theta}
-
\int
q(\theta_{i})\cdot q(\theta_{-i})
\log \pi(\bm{\theta}|\textbf{y}))
d\bm{\theta}.
\end{align}
The first integral on the right-hand side of (\ref{eq:variational_factor_proof_1}) can be further splitted because two densities, $q(\theta_{i}) \in \mathcal{Q}_{\theta_{i}}^{m}$ and $q(\theta_{-i})\in \mathcal{Q}_{\theta_{-i}}^{MF}$ (\ref{eq:the_i_th_complemental_variational_family}), are independent by the mean-field assumption (\ref{eq:mean-field variational family})
\begin{align}
\label{eq:variational_factor_proof_2}
\int
q(\theta_{i})\cdot q(\theta_{-i})
\log \{q(\theta_{i}) \cdot q(\theta_{-i})\}
d\bm{\theta}
=
\int
q(\theta_{i}) \log q(\theta_{i})
d\theta_{i}
+
\int
q(\theta_{-i})
\log q(\theta_{-i})
d\theta_{-i}.
\end{align}
Note that the second integral on the right-hand side of (\ref{eq:variational_factor_proof_2}) is constant with respect to $q(\theta_{i}) \in \mathcal{Q}_{\theta_{i}}^{m}$.

\noindent The second integral on the right-hand side of (\ref{eq:variational_factor_proof_1}) can be further splitted by using $\pi(\bm{\theta}|\textbf{y}) = \pi(\theta_{i}|\theta_{-i},\textbf{y})
\cdot
\pi(\theta_{-i}|\textbf{y})$
\begin{align}
\nonumber
&\int
q(\theta_{i})\cdot q(\theta_{-i})
\log \pi(\bm{\theta}|\textbf{y}))
d\bm{\theta}
\\
\label{eq:variational_factor_proof_3}
&
\,\,=
\int
q(\theta_{i})\cdot q(\theta_{-i})
\log 
\pi(\theta_{i}|\theta_{-i},\textbf{y})
d\bm{\theta}
+
\int
q(\theta_{i})\cdot q(\theta_{-i})
\log 
\pi(\theta_{-i}|\textbf{y})
d\bm{\theta}.
\end{align}
\noindent Note that the second integral on the right-hand side of (\ref{eq:variational_factor_proof_3}) is constant with respect to $q(\theta_{i}) \in \mathcal{Q}_{\theta_{i}}^{m}$. This is because we have
$\int
q(\theta_{i}) \cdot q(\theta_{-i})
\log 
\pi(\theta_{-i}|\textbf{y})
d\bm{\theta}
$
$
=
\int \{\int
q(\theta_{i}) \cdot q(\theta_{-i})
\log 
\pi(\theta_{-i}|\textbf{y})
d\theta_{-i}
\}
d\theta_{i}
$
$
=
\int
q(\theta_{i})
\{
\int
q(\theta_{-i})
$ 
$
\log 
\pi(\theta_{-i}|\textbf{y})
d\theta_{-i}
\}
d\theta_{i}$
$=
\int
q(\theta_{-i})
\log 
\pi(\theta_{-i}|\textbf{y})
$
$
d\theta_{-i}
\cdot
\int
q(\theta_{i})
d\theta_{i}
=
\int
q(\theta_{-i})
\log 
\pi(\theta_{-i}|\textbf{y})
$
$
d\theta_{-i}
$, which is independent of $q(\theta_{i})$.

\noindent
Let us introduce a function $\nu(\theta_{i}) = \exp \mathbb{E}_{q(\theta_{-i})}[\log \pi(\theta_{i} |\theta_{-i},\textbf{y})]$ on $\Theta_{i}$. The function $\nu(\theta_{i})$ is not a density, but after normalization, the function $\nu(\theta_{i})/\int \nu(\theta_{i}) d\theta_{i} \in \mathcal{Q}_{\theta_{i}|\textbf{y}}^{m}$ becomes a marginal density supported on $\Theta_{i}$. Now, we can simplify the first integral on the right-hand side of (\ref{eq:variational_factor_proof_3}) as follows 
\begin{align}
\nonumber
\int
q(\theta_{i})\cdot q(\theta_{-i})
\log 
\pi(\theta_{i}|\theta_{-i},\textbf{y})
d\bm{\theta}
&=
\int
q(\theta_{i}) \cdot 
\log
\nu(\theta_{i})
d\theta_{i}
\\
\label{eq:variational_factor_proof_4}
&=
\int
q(\theta_{i}) \cdot 
\log
\bigg(\frac{\nu(\theta_{i})}{\int \nu(\theta_{i}) d\theta_{i}}
\bigg)
d\theta_{i}
+
\log \bigg(\int \nu(\theta_{i}) d\theta_{i}\bigg).
\end{align}
\noindent Note that the second integral on the right-hand side of (\ref{eq:variational_factor_proof_4}) is constant with respect to $q(\theta_{i}) \in \mathcal{Q}_{\theta_{i}}^{m}$.

\noindent By using the derived results (\ref{eq:variational_factor_proof_2}) -- (\ref{eq:variational_factor_proof_4}), we can re-express the $\text{KL}(q(\theta_{i})\cdot q(\theta_{-i})||
\pi(\bm{\theta}|\textbf{y}))$ (\ref{eq:variational_factor_proof_1}) as follows
\begin{align*}
\text{KL}(q(\theta_{i})\cdot q(\theta_{-i})||
\pi(\bm{\theta}|\textbf{y}))
&=
\int
q(\theta_{i}) 
\log \bigg(
\frac{q(\theta_{i}) }{
\nu(\theta_{i})/\int\nu(\theta_{i})d\theta_{i}
}
\bigg)
d\theta_{i}
+ \text{const w.r.t } q(\theta_{i})\\
&=
\text{KL}
\bigg(
q(\theta_{i}) 
\bigg|\bigg|
\frac{\nu(\theta_{i})}{\int\nu(\theta_{i})d\theta_{i}}
\bigg) + \text{const w.r.t } q(\theta_{i}).
\end{align*}
The above equality means that for a fixed $q(\theta_{-i}) \in \mathcal{Q}_{\theta_{-i}}^{MF}$,  minimizing the $\text{KL}(q(\theta_{i})\cdot q(\theta_{-i})||
\pi(\bm{\theta}|\textbf{y}))$ with respect to $q(\theta_{i}) \in \mathcal{Q}_{\theta_{i}}^{m}$ is equivalent with minimizing the $\text{KL}
(q(\theta_{i}) ||$ $
\nu(\theta_{i})/\int\nu(\theta_{i})d\theta_{i})$ with respect to $q(\theta_{i}) \in \mathcal{Q}_{\theta_{i}}^{m}$. The latter one vanishes if and only if it holds $q(\theta_{i}) = \nu(\theta_{i})/\int\nu(\theta_{i})d\theta_{i}$.
\end{proof}

%https://www.maths.usyd.edu.au/u/jormerod/JTOpapers/Ormerod10.pdf
\begin{algorithm}[h]
\caption{CAVI algorithm}\label{alg:CAVI algorithm}
\SetAlgoLined
\textbf{Initialize:} pick arbitrary starting variational density
$q(\bm{\theta}) = q(\theta_{1}) q(\theta_{2})\cdots q(\theta_{K})
\in \mathcal{Q}_{\bm{\theta}}^{MF}$ \\
$\quad\quad\quad\quad$ denote $q^{*}(\theta_{1}) \leftarrow q(\theta_{1})$, $q^{*}(\theta_{2}) \leftarrow q(\theta_{2})$, $\cdots$, $q^{*}(\theta_{K}) \leftarrow q(\theta_{K})$
\\
\textbf{Iterate a cycle:}\\
$\quad$ $\emph{Step 1.}$ $\text{set}$ $
q^{*}(\theta_{1}) \leftarrow 
\dfrac{\exp\ \mathbb{E}_{q^{*}(\theta_{-1})}[  \log\ \pi(\theta_{1}|\theta_{-1},\textbf{y}) ]}{\int \exp\ \mathbb{E}_{q^{*}(\theta_{-1})}[  \log\ \pi(\theta_{1}|\theta_{-1},\textbf{y}) ] d\theta_{1}}
\in \mathcal{Q}_{\theta_{1}|\textbf{y}}^{m}
$\\
$\quad$ $\emph{Step 2.}$ $\text{set}$ $
q^{*}(\theta_{2}) \leftarrow 
\dfrac{\exp\ \mathbb{E}_{q^{*}(\theta_{-2})}[  \log\ \pi(\theta_{2}|\theta_{-2},\textbf{y}) ]}{\int \exp\ \mathbb{E}_{q^{*}(\theta_{-2})}[  \log\ \pi(\theta_{2}|\theta_{-2},\textbf{y}) ] d\theta_{2}}
\in \mathcal{Q}_{\theta_{2}|\textbf{y}}^{m}
$\\
$\quad\quad\quad\quad\quad\quad\quad\quad$ $\vdots$\\
$\quad$ $\emph{Step K.}$  $\text{set}$ $
q^{*}(\theta_{K}) \leftarrow 
\dfrac{\exp\ \mathbb{E}_{q^{*}(\theta_{-K})}[  \log\ \pi(\theta_{K}|\theta_{-K},\textbf{y}) ]}{\int \exp\ \mathbb{E}_{q^{*}(\theta_{-K})}[  \log\ \pi(\theta_{K}|\theta_{-K},\textbf{y}) ] d\theta_{K}}
\in \mathcal{Q}_{\theta_{K}|\textbf{y}}^{m}
$\\
\textbf{end Iterate}
\end{algorithm}

Algorithm \ref{alg:CAVI algorithm} details a generic CAVI algorithm. Broadly speaking, when we implement the CAVI algorithm for a Bayesian model $\{p(\textbf{y}|\bm{\theta}),\pi(\bm{\theta})\}$, having specified a certain mean-field variational family $\mathcal{Q}_{\bm{\theta}}^{MF}$ (\ref{eq:mean-field variational family}), we wish two consequences: (a) the VB posterior $q^{*}(\bm{\theta}) = \prod_{i=1}^{K}q^{*}(\theta_{i}) \in \mathcal{Q}_{\bm{\theta}|\textbf{y}}^{MF}$ (\ref{eq:mean-field variational family_a_posteriori}) nicely approximates the global minimizer $\hat{q}^{MF}(\bm{\theta})$ (\ref{eq:variational_density_general}); and then (b) the global minimizer $\hat{q}^{MF}(\bm{\theta})$ nicely approximates our target density $\pi(\bm{\theta}|\textbf{y})$. Most of research articles \citep{bickel2013asymptotic,celisse2012consistency,wang2019frequentist,you2014variational} assume that the global minimum (\ref{eq:variational_density_general}) can be achieved and work on theoretical aspects of the global minimizer $\hat{q}^{MF}(\bm{\theta})$ (\ref{eq:variational_density_general}). On the other hand, there are only a few research works \citep{Zhang_AOS} that directly investigate theoretical aspects of an iterative algorithm. In this paper, we directly study the CAVI algorithm (Algorithm \ref{alg:CAVI algorithm}), and explain how key ingredients used in CAVI are functionally related each other by the duality formula (\ref{eq:duality_formula_continuous_version}).

We convey two messages. First, the full conditional posterior $\pi(\theta_{i}|\theta_{-i},\textbf{y}) \in \mathcal{Q}_{\theta_{i}|\textbf{y}}$ (\ref{eq:full_conditional_for_theta_i}) plays a central role in the updating procedures not only for the Gibbs sampler but also for the CAVI algorithm \citep{ormerod2010explaining}. Second, although the Gibbs sampler eventually leads to the exact target density $\pi(\bm{\theta}|\textbf{y}) \in \mathcal{Q}_{\bm{\theta}|\textbf{y}}$ (\ref{eq:posterior_distribution}) when the number of iterations goes to infinity under reasonably general conditions \citep{roberts1994simple}, this is not guaranteed for the MFVI. Set-theoretically, the later is obvious due to the nonvanishing divergence $\text{KL}(q(\bm{\theta})||\pi(\bm{\theta}|\textbf{y})) > 0$ (\ref{eq:variational_density_general}) induced by the proper subset relationship $\mathcal{Q}_{\bm{\theta}|\textbf{y}}^{MF}\subsetneq \mathcal{Q}_{\bm{\theta}|\textbf{y}}$: refer to Figure \ref{fig:set_inclusion}. 
\section{Gibbs sampler revisited by the duality formula}\label{sec:Gibbs sampler revisited by the duality formula}
The Gibbs transitional kernel (\ref{eq:transition_kernel_Gibbs}) explains the stationary movement of a state $\bm{\theta}\in \bm{\Theta}$ from the present cycle to the next cycle (i.e., longitudinal point of view on the state). In this Section, our focus is to explain within a cycle how the ingredients of the Gibbs samplers, that is, $\pi(\theta_{i}|\theta_{-i},\textbf{y}) \in \mathcal{Q}_{\theta_{i}|\textbf{y}}$, $\pi(\theta_{-i}|\textbf{y})\in \mathcal{Q}_{\theta_{-i}|\textbf{y}}$, and $\pi(\bm{\theta}|\textbf{y}) \in \mathcal{Q}_{\bm{\theta}|\textbf{y}}$, are \emph{functionally} associated each other (i.e., cross-sectional point of view on the densities):
\begin{corollary}\label{cor:Gibbs_sampler_revisited} 
Consider a Bayesian model $\{p(\textbf{y}|\bm{\theta})$,$\pi(\bm{\theta})\}$ with the entire parameter space $\bm{\Theta}$ decomposed as (\ref{eq:decomposition_of_entire_ps}). For each $i$ ($i=1, \cdots,K$), define a functional $\mathcal{F}_{i}: \mathcal{Q}_{\theta_{i}} \rightarrow \mathbb{R}$ induced by the duality formula as follow:
\begin{align}
\label{eq:functional_theta_i}
\mathcal{F}_{i}\{q(\theta_{i})\}
&\defeq 
\mathbb{E}_{q(\theta_{i})}[\log \pi(\theta_{-i}|\theta_{i}, \textbf{y}) ] - \text{KL}(q(\theta_{i})||\pi(\theta_{i}|\textbf{y}) ).
\end{align}
Then the following relations hold.
\begin{itemize}
\item[]\textbf{(a)} The functional $\mathcal{F}_{i}$ is concave over $\mathcal{Q}_{\theta_{i}}$.
\item[]\textbf{(b)} For all densities $q(\theta_{i}) \in \mathcal{Q}_{\theta_{i}|\textbf{y}}$,
\begin{align*}
\mathcal{F}_{i}\{q(\theta_{i})\} \leq \log \pi(\theta_{-i}|\textbf{y}).
\end{align*}
\item[]\textbf{(c)} The functional $\mathcal{F}_{i}$ attains the value $\log \pi(\theta_{-i}|\textbf{y})$ only at the full conditional posterior density $q(\theta_{i})=\pi(\theta_{i}|\theta_{-i},\textbf{y}) \in \mathcal{Q}_{\theta_{i}|\textbf{y}}$ (\ref{eq:full_conditional_for_theta_i}).
\end{itemize}
\end{corollary}

\begin{proof}
\textbf{\emph{(a)}} Let $p(\theta_{i})$ and $q(\theta_{i})$ are elements of the set $\mathcal{Q}_{\theta_{i}}$. For any $ 0 \leq a \leq 1$, we have
\begin{align}
\nonumber
&\mathcal{F}_{i}\{
a
p(\theta_{i})
+
(1 -a )
q(\theta_{i})
 \}\\
 \label{eq:cor_Gibbs_a_pr_2}
&\,\, =
 \int \log \pi(\theta_{-i}|\theta_{i},\textbf{y}) 
 \{
a 
  p(\theta_{i})
+(1 -a )
q(\theta_{i})
 \}
 d\theta_{i}-
 \text{KL}(
a 
  p(\theta_{i})
+
(1 -a )
q(\theta_{i})
||
\pi(\theta_{i}|\textbf{y}) 
 ).
\end{align}
The first term on the right-hand side of (\ref{eq:cor_Gibbs_a_pr_2}) can be written as

\begin{align}
\nonumber
  &\int \log \pi(\theta_{-i}|\theta_{i},\textbf{y}) 
 \{
a 
  p(\theta_{i})
+(1 -a )
q(\theta_{i})
 \}
 d\theta_{i}
\\
 \label{eq:cor_Gibbs_a_pr_3}
&\,\,  =
a \mathbb{E}_{p(\theta_{i})}[\log \pi(\theta_{-i}|\theta_{i}, \textbf{y}) ]
+
(1-a) \mathbb{E}_{q(\theta_{i})}[\log \pi(\theta_{-i}|\theta_{i}, \textbf{y}) ],
\end{align}
where the expectation $\mathbb{E}_{p(\theta_{i})}[\cdot]$ and $\mathbb{E}_{q(\theta_{i})}[\cdot]$ are taken with respect to the densities $p(\theta_{i})$ and $q(\theta_{i})$, respectively.

The (negative of) second term on the right-hand side (\ref{eq:cor_Gibbs_a_pr_2}) satisfies the following inequality
\begin{align}
 \label{eq:cor_Gibbs_a_pr_4}
&\text{KL}(
a 
  p(\theta_{i})
+
(1 -a )
q(\theta_{i})
||
\pi(\theta_{i}|\textbf{y}) 
 )
\leq
a 
 \text{KL}(
  p(\theta_{i})
||
\pi(\theta_{i}|\textbf{y}) 
 )
 +
 (1-a)
  \text{KL}(
  q(\theta_{i})
||
\pi(\theta_{i}|\textbf{y}) 
 ).
\end{align}
The inequality (\ref{eq:cor_Gibbs_a_pr_4}) generally holds due to the joint convexity of the $f$-divergence; see Lemma 4.1 of \citep{csiszar2004information}.

Now, use the expression (\ref{eq:cor_Gibbs_a_pr_3}) and inequality (\ref{eq:cor_Gibbs_a_pr_4}) to finish the proof:
\begin{align*}
\mathcal{F}_{i}\{
a
p(\theta_{i})
+
(1 -a )
q(\theta_{i})
 \}
& \geq
 a
\{
 \mathbb{E}_{p(\theta_{i})}[\log \pi(\theta_{-i}|\theta_{i}, \textbf{y}) ]
-
  \text{KL}(
  p(\theta_{i})
||
\pi(\theta_{i}|\textbf{y}) )
\} 
\\
&+
(1-a)
\{
 \mathbb{E}_{q(\theta_{i})}[\log \pi(\theta_{-i}|\theta_{i}, \textbf{y}) ]
 -
   \text{KL}(
  p(\theta_{i})
||
\pi(\theta_{i}|\textbf{y}) )
 \} 
\\ 
 &=
 a
 \mathcal{F}_{i}\{
p(\theta_{i})
 \}
 +
 (1-a)
  \mathcal{F}_{i}\{
q(\theta_{i})
 \}.
 \end{align*}
\noindent \textbf{\emph{ (b) and (c)}} For each $i=1, \cdots, K$, use the duality formula (\ref{eq:duality_formula_continuous_version}) by replacing the $q(\theta)$, $p(\theta)$, and $h(\theta)$ in the formula with $q(\theta_{i}) \in \mathcal{Q}_{\theta_{i}}$, $\pi(\theta_{i}|\textbf{y})\in \mathcal{Q}_{\theta_{i}|\textbf{y}}$, and $\log \pi(\theta_{-i}|\theta_{i},\textbf{y})\in \mathcal{Q}_{\theta_{-i}|\textbf{y}}$, respectively. (Recall that in the formula (\ref{eq:duality_formula_continuous_version}), $q$ and $p$ need to be densities, whereas $h$ is a measurable function.) This leads to the following equality
\begin{align}
\label{eq:Gibbs_sampler_derivation_pf_1}
&\log\eE_{\pi(\theta_{i}|\textbf{y})}[\pi(\theta_{-i}|\theta_{i},\textbf{y})]
= \text{sup}_{q(\theta_{i}) \ll \pi(\theta_{i}|\textbf{y})} \{
\eE_{q(\theta_{i})}[\log \pi(\theta_{-i}|\theta_{i},\textbf{y})]-
\KL(q(\theta_{i}) \|\pi(\theta_{i}|\textbf{y})) \},
\end{align}
where the supremum on the right-hand side of (\ref{eq:Gibbs_sampler_derivation_pf_1}) is attained if and only if
\begin{align*}
q(\theta_{i}) &= \pi(\theta_{i}|\textbf{y}) \cdot \frac{\pi(\theta_{-i}|\theta_{i},\textbf{y})}{\eE_{\pi(\theta_{i}|\textbf{y})}[\pi(\theta_{-i}|\theta_{i},\textbf{y})]} = \pi(\theta_{i}|\theta_{-i},\textbf{y}) \in \mathcal{Q}_{\theta_{i}|\textbf{y}}.
\end{align*}
\noindent On the other hand, it is straightforward to derive that the left-hand side of (\ref{eq:Gibbs_sampler_derivation_pf_1}) is $\log\eE_{\pi(\theta_{i}|\textbf{y})}[\pi(\theta_{-i}|\theta_{i},\textbf{y})]=\log \pi(\theta_{-i}|\textbf{y})$.

\noindent Finalize the proof by using the above facts: by (\ref{eq:Gibbs_sampler_derivation_pf_1}), it holds the following inequality
\begin{align}
\nonumber
\mathcal{F}_{i}\{q(\theta_{i})\}=
\eE_{q(\theta_{i})}[\log \pi(\theta_{-i}|\theta_{i},\textbf{y})]-
\KL(q(\theta_{i}) \|\pi(\theta_{i}|\textbf{y})) \leq \log \pi(\theta_{-i}|\textbf{y})
\end{align}
for all density $q(\theta_{i})$ supported on $\Theta_{i}$ which satisfies the dominating condition $q(\theta_{i}) \ll \pi(\theta_{i}|\textbf{y})$, where the equality holds if and only if $q(\theta_{i}) = \pi(\theta_{i}|\theta_{-i},\textbf{y}) \in \mathcal{Q}_{\theta_{i}|\textbf{y}}$. 
%Here, we used the definition of the notations $q \ll p$ and $\mathcal{Q}_{\theta_{i}|\textbf{y}}$ to conclude.
\end{proof}
Corollary \ref{cor:Gibbs_sampler_revisited} states the full conditional posterior distribution $\pi(\theta_{1}|\theta_{-1},\textbf{y}) \in \mathcal{Q}_{\theta_{1}|\textbf{y}}$ (and likewisely $\pi(\theta_{2}|\theta_{-2},\textbf{y})\in \mathcal{Q}_{\theta_{2}|\textbf{y}}$, $\cdots$, $\pi(\theta_{K}|\theta_{-K},\textbf{y})\in \mathcal{Q}_{\theta_{K}|\textbf{y}}$) is the global maximum of the induced functional $\mathcal{F}_{1}$ (\ref{eq:functional_theta_i}) (and likewisely $\mathcal{F}_{2}$, $\cdots$, $\mathcal{F}_{K}$) with the corresponding maximum value to be $\log \pi(\theta_{-1}|\textbf{y}) \in \mathbb{R}$ (and likewisely $\log \pi(\theta_{-2}|\textbf{y})\in \mathbb{R}$, $\cdots$, $\log \pi(\theta_{-K}|\textbf{y})\in \mathbb{R}$). See Figure \ref{fig:Gibbs_sampler_picture} for a pictorial description of Corollary \ref{cor:Gibbs_sampler_revisited}. 
\begin{figure}[h]
\centering
\includegraphics[width=1\textwidth]{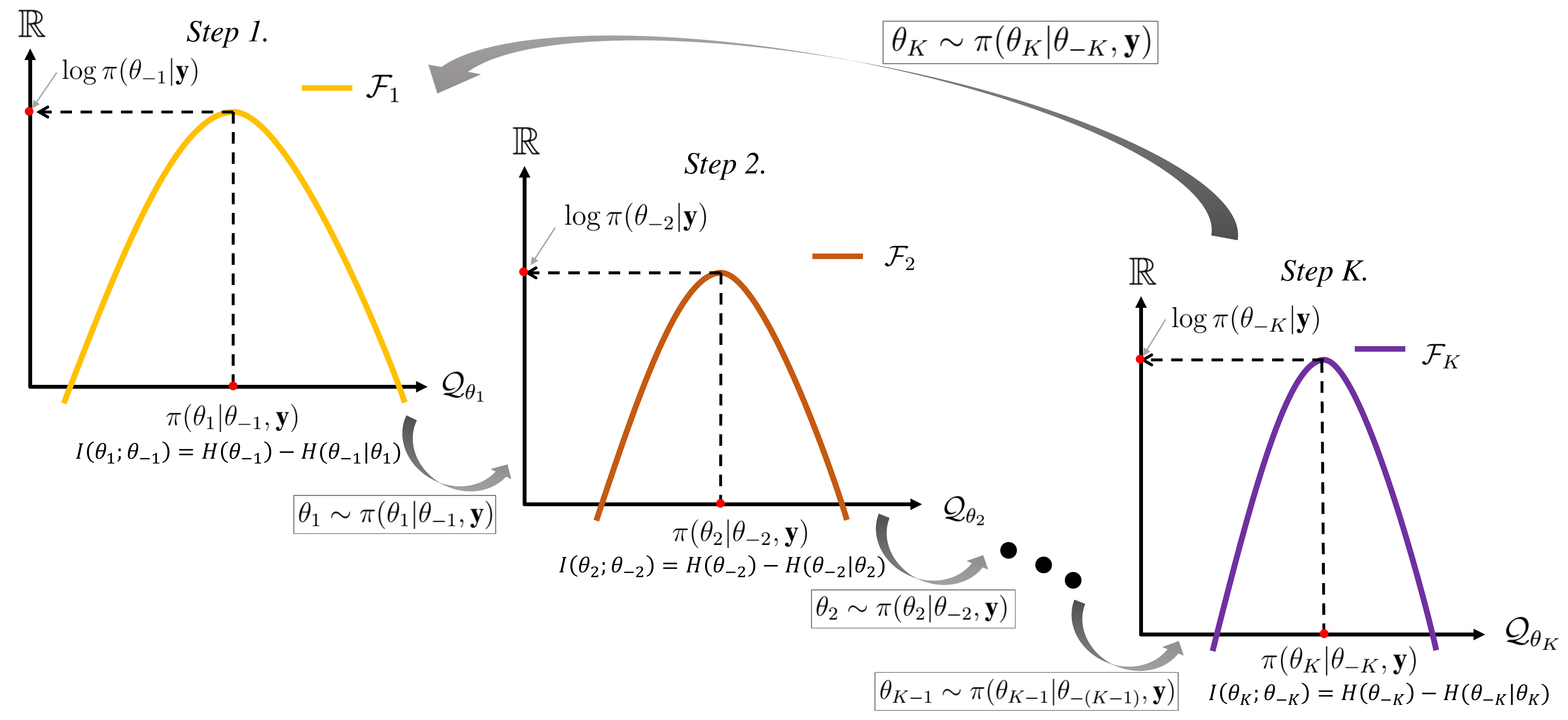}
\caption{Pictorial description of Gibbs sampler (Algorithm \ref{alg:Gibbs sampler}). At the $i$-th step within a cycle ($i=1,\cdots,K$), the panel shows that the $i$-th full conditional posterior $\pi(\theta_{i}|\theta_{-i},\textbf{y})$ (\ref{eq:full_conditional_for_theta_i}) is the global maximum of the induced functional $\mathcal{F}_{i}$ (\ref{eq:functional_theta_i}) with $\log \pi(\theta_{-i}|\textbf{y})$ as the maximum value. They are related with the information equality (\ref{eq:Gibbs_information_equality}).}
\label{fig:Gibbs_sampler_picture}
\end{figure}

We can derive information equations governing the Gibbs sampler by using the Corollary \ref{cor:Gibbs_sampler_revisited}. These equations tell us how much amount information is transmitted at each step within a cycle.
\begin{corollary}\label{corollary:Information_theory_Gibbs_sampler}Consider a Bayesian model $\{p(\textbf{y}|\bm{\theta})$,$\pi(\bm{\theta})\}$ with the entire parameter space $\bm{\Theta}$ decomposed as (\ref{eq:decomposition_of_entire_ps}). For each $i$ ($i=1,\cdots,K$), define the posterior mutual information of $(\theta_{i},\theta_{-i}) \in \Theta_{i} \times \Theta_{-i}$, posterior differential entropy of $\theta_{-i} \in \Theta_{-i}$, and posterior conditional differential entropy of $\theta_{-i} \in \Theta_{-i}$ given $\theta_{i} \in \Theta_{i}$ as follows:
\begin{align}
\label{eq:mutual_information}
%%%%%%%%%%%%%%%%%%%%%%%%%%%%%%%%
I(\theta_{i} ; \theta_{-i} )&\defeq \KL(\pi(\bm{\theta}|\textbf{y})
||
\pi(\theta_{i}|\textbf{y}) \pi(\theta_{-i}|\textbf{y}))
=
\int_{\bm{\Theta}} 
\log\ \bigg(
\frac{\pi(\bm{\theta}|\textbf{y})}{\pi(\theta_{i}|\textbf{y}) \pi(\theta_{-i}|\textbf{y})}
\bigg)
\pi(\bm{\theta}|\textbf{y})
d\bm{\theta},
\\
\label{eq:entropy}
%%%%%%%%%%%%%%%%%%%%%%%%
H(\theta_{-i})&\defeq -\mathbb{E}_{\pi(\theta_{-i}|\textbf{y})}[\log \pi(\theta_{-i}|\textbf{y})]
=
-
\int_{\Theta_{-i}}
\log \pi(\theta_{-i}|\textbf{y})
\pi(\theta_{-i}|\textbf{y}) d\theta_{-i},\\
\label{eq:conditional entropy}
%%%%%%%%%%%%%%%%%%%%%%%%
H(\theta_{-i}|\theta_{i})&\defeq 
-\mathbb{E}_{\pi(\bm{\theta}|\textbf{y})}[\log \pi(\theta_{-i}|\theta_{i},\textbf{y})]
=
-\int_{\bm{\Theta}}
\log \pi(\theta_{-i}|\theta_{i}, \textbf{y})
\pi(\bm{\theta}|\textbf{y})d\bm{\theta}.
\end{align}
Then the following $K$ information equations hold
\begin{align}
\label{eq:Gibbs_information_equality}
I(\theta_{i};\theta_{-i}) &= H(\theta_{-i}) - H(\theta_{-i}|\theta_{i}), \quad (i=1,\cdots,K).
\end{align}
\end{corollary}
\begin{proof}
% Proof notaition exactly follow: https://en.wikipedia.org/wiki/Mutual_information
By Corollary \ref{cor:Gibbs_sampler_revisited} \textbf{(b)} and \textbf{(c)}, for each $i$ ($i=1,\cdots, K$) the following equality holds
\begin{align*}
\KL(\pi(\theta_{i}|\theta_{-i},\textbf{y}) \|\pi(\theta_{i}|\textbf{y})) =
-
\log \pi(\theta_{-i}|\textbf{y})
+\eE_{\pi(\theta_{i}|\theta_{-i},\textbf{y})}[\log \pi(\theta_{-i}|\theta_{i},\textbf{y})].
\end{align*}
Take the integral $\mathbb{E}_{\pi(\theta_{-i}|\textbf{y})}[\cdot] =\int_{\Theta_{-i}} \cdot \, \pi(\theta_{-i}|\textbf{y})d\theta_{-i} $ on both sides of the above equality to have
\begin{align*}
&\mathbb{E}_{\pi(\theta_{-i}|\textbf{y})}[
\KL(\pi(\theta_{i}|\theta_{-i},\textbf{y}) \|\pi(\theta_{i}|\textbf{y}))]
\\
&\,\,=
-
\mathbb{E}_{\pi(\theta_{-i}|\textbf{y})}[
\log \pi(\theta_{-i}|\textbf{y})]
-
(
-
\mathbb{E}_{\pi(\theta_{-i}|\textbf{y})}[
\eE_{\pi(\theta_{i}|\theta_{-i},\textbf{y})}[\log \pi(\theta_{-i}|\theta_{i},\textbf{y})]
]
),
\end{align*}
which is the same as $I(\theta_{i} ; \theta_{-i})
=
H(\theta_{-i})
-
H(\theta_{-i}|\theta_{i})
$. 
\end{proof}
Figure \ref{fig:Gibbs_Channel} displays a schematic description of Corollary \ref{corollary:Information_theory_Gibbs_sampler}. For each $i$ ($i=1,\cdots,K$), we define $I(\theta_{-i} ; \theta_{i} )$, $H(\theta_{i})$, and $H(\theta_{i}|\theta_{-i})$ by interchanging the $i$ and $-i$ from the $I(\theta_{i} ; \theta_{-i} )$ (\ref{eq:mutual_information}), $H(\theta_{-i})$ (\ref{eq:entropy}), and $H(\theta_{-i}|\theta_{i})$ (\ref{eq:conditional entropy}), respectively.
Then we can show that it holds $ I(\theta_{i} ; \theta_{-i} )=H(\theta_{-i}) - H(\theta_{-i}|\theta_{i}) =  H(\theta_{i}) - H(\theta_{i}|\theta_{-i}) =I(\theta_{-i} ; \theta_{i})$ (Theorem 2.4.1 of \citep{cover1999elements}). Mutual information $I(\theta_{i} ; \theta_{-i} )$ quantifies the reduction in uncertainty of random quantity $\theta_{i} \in \Theta_{i}$ once we know the other $\theta_{-i}\in \Theta_{-i}$, \emph{a posteriori}. It vanishes if and only if $\theta_{i}$ and $\theta_{-i}$ are marginally independent, \emph{a posteriori}. Within a cycle of the Gibbs sampler (Algorithm \ref{alg:Gibbs sampler}), this mutual information is transmitted through an iterative Monte Carlo scheme \citep{gelfand1990sampling}. 
\begin{figure}[h]
\centering
\includegraphics[width=0.8\textwidth]{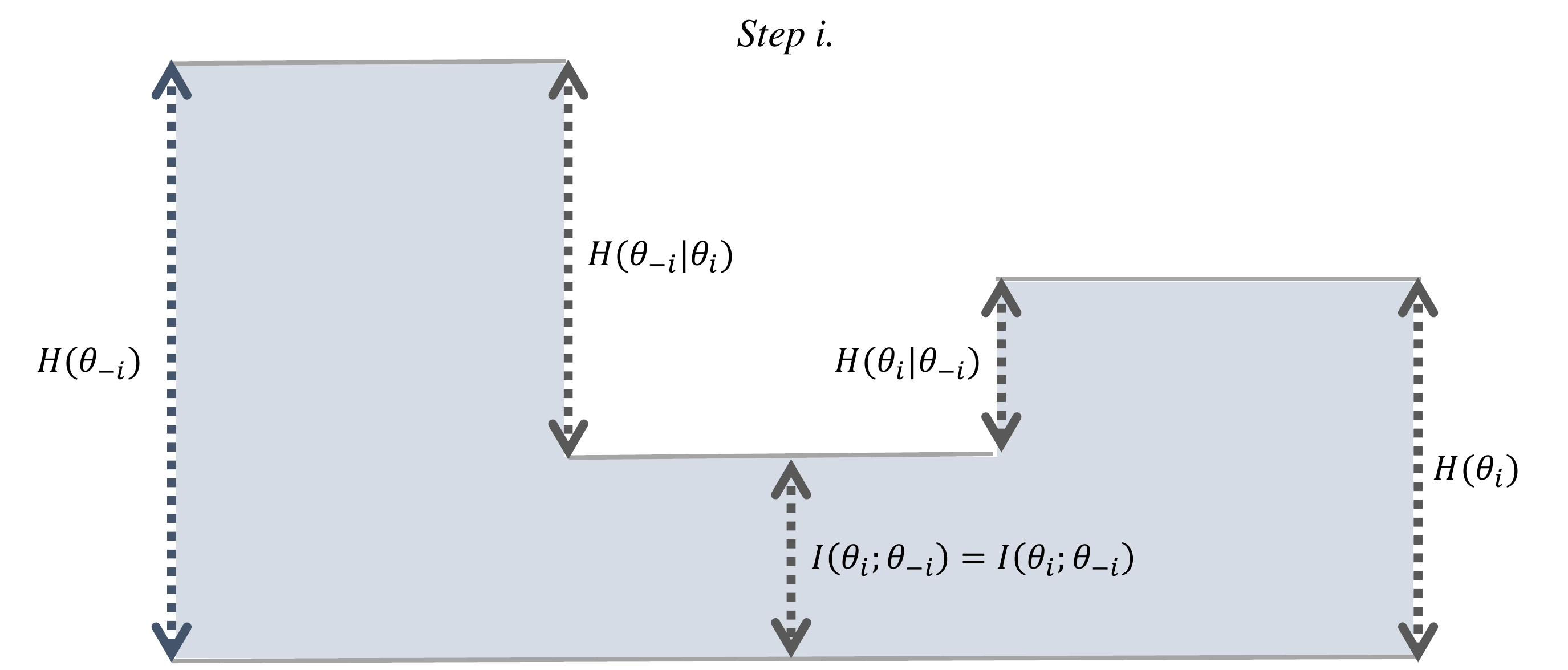}
\caption{Schematic description of the information equality associated with the Gibbs sampler (Algorithm \ref{alg:Gibbs sampler}). At the $i$-th step ($i=1,\cdots,K$) within a cycle, it holds $I(\theta_{i} ; \theta_{-i} )=H(\theta_{-i}) - H(\theta_{-i}|\theta_{i}) =  H(\theta_{i}) - H(\theta_{i}|\theta_{-i}) =I(\theta_{-i} ; \theta_{i})$. The quantity $I(\theta_{i} ; \theta_{-i} )$ can be interpreted as the amount of information that shall be transmitted from the $i$-step to the $(i+1)$-th step within a cycle via drawing a sample $\theta_{i} \sim \pi(\theta_{i}|\theta_{-i},\textbf{y})$.}
\label{fig:Gibbs_Channel}
\end{figure}
\section{CAVI algorithm revisited by the duality formula}\label{sec:CAVI algorithm revisited by the duality formula}
Consider a Bayesian model $\{p(\textbf{y}|\bm{\theta})$,$\pi(\bm{\theta})\}$ with a mean-field variational family $\mathcal{Q}_{\bm{\theta}}^{MF}$ (\ref{eq:mean-field variational family}). Having employed the CAVI algorithm (Algorithm \ref{alg:CAVI algorithm}), for each $i$ ($i=1,\cdots,K$), we can obtain the analytic formula of the $i$-th variational factor $q^{*}(\theta_{i})$ given as (\ref{eq:variational_factor_theta_i}). We can regard this variational factor $q^{*}(\theta_{i})$ as a surrogate for the marginal posterior density $\pi(\theta_{i}|\textbf{y})$. Note that the two densities belong to the same set $\mathcal{Q}_{\theta_{i}|\textbf{y}}^{m}$ (defined in the item $\textbf{(iv)}$); refer to the Venn diagram in Figure \ref{fig:set_inclusion}. This suggests that an `intrinsic' approximation quality due to the CAVI algorithm can be explained by the Kullback-Leibler divergence $\text{KL}(q^{*}(\theta_{i})||\pi(\theta_{i}|\textbf{y}))$ or its lower bound: a lower value may indicate a better approximation quality, which can be further applied to diagnostic of the algorithm \citep{yao2018yes}.

In practice, although it is possible to sample from marginal posterior density $\pi(\theta_{i}|\textbf{y})$ ($i=1,\cdots,K$) through various MCMC techniques \citep{gelman2013bayesian}, but it is difficult to obtain an analytic expression of the density $\pi(\theta_{i}|\textbf{y})$, hence, so is for the divergence $\text{KL}(q^{*}(\theta_{i})||\pi(\theta_{i}|\textbf{y}))$. It is also nontrivial to acquire a lower bound for $\text{KL}(q^{*}(\theta_{i})||\pi(\theta_{i}|\textbf{y}))$ through information inequalities (for example, Pinsker's inequality \citep{massart2007concentration}) as such inequalities again require a closed-form expression for the density $\pi(\theta_{i}|\textbf{y})$ for each $i$, ($i=1,\cdots,K$).

The duality formula (\ref{eq:duality_formula_continuous_version}) provides some heuristic insight about how the two densities, $q^{*}(\theta_{i})$ and $\pi(\theta_{i}|\textbf{y})$, are related, and an algorithmic-based lower bound for the $\text{KL}(q^{*}(\theta_{i})||$
$\pi(\theta_{i}|\textbf{y}))$ \emph{without requiring} analytical expression of $\pi(\theta_{i}|\textbf{y})$ when the CAVI algorithm (Algorithm \ref{alg:CAVI algorithm}) is employed:
\begin{corollary}\label{cor:CAVI_algorithm_revisited} 
Consider a Bayesian model $\{p(\textbf{y}|\bm{\theta})$,$\pi(\bm{\theta})\}$ with the entire parameter space $\bm{\Theta}$ decomposed as (\ref{eq:decomposition_of_entire_ps}). Provided mean-field variational family $\mathcal{Q}_{\bm{\theta}}^{MF}$ (\ref{eq:mean-field variational family}), assume that the CAVI algorithm (Algorithm \ref{alg:CAVI algorithm}) is employed to approximate the target density $\pi(\bm{\theta}|\textbf{y})$ (\ref{eq:posterior_distribution}). For each $i$ ($i=1,\cdots,K$), define a functional $\mathcal{R}_{-i}:\mathcal{Q}_{\theta_{-i}}^{MF} \rightarrow (0, \infty)$ induced by the duality formula as follow:
\begin{align}
\label{eq:squashing_functional}
\mathcal{R}_{-i}\{q(\theta_{-i})\} \defeq  \frac{\int \exp \eE_{q(\theta_{-i})}[\log \pi(\theta_{i}|\theta_{-i},\textbf{y})] d\theta_{i} }{\exp \text{KL}(q(\theta_{-i})|| \pi(\theta_{-i}|\textbf{y}))}.
\end{align}

\noindent Then the followings relations hold.
\begin{itemize}
\item[]\textbf{(a)} There exists $\mathcal{R}_{-i}\{ q^{*}(\theta_{-i}) \}  \in (0,1]$ which is constant with respect to $q^{*}(\theta_{i})$ such that
\begin{align}
\label{eq:squashing_inequality}
\mathcal{R}_{-i}\{ q^{*}(\theta_{-i}) \} \cdot q^{*}(\theta_{i}) \leq \pi(\theta_{i}|\textbf{y})\quad \text{for all } \theta_{i} \in \Theta_{i}.
\end{align}
The value $\mathcal{R}_{-i}\{ q^{*}(\theta_{-i}) \}$ is called the squashing constant for the variational factor $q^{*}(\theta_{i})$.
\item[]\textbf{(b)} Kullback-Leibler divergence between $q^{*}(\theta_{i})$ and $\pi(\theta_{i}|\textbf{y})$ is lower bounded by
\begin{align}
\label{eq:KL_lower_bound}
\text{KL}(q^{*}(\theta_{i})||\pi(\theta_{i}|\textbf{y})) 
\geq \text{max}\bigg\{0,
\log \bigg(
\int \exp\ \mathbb{E}_{q^{*}(\theta_{i})}[  \log\ \pi(\theta_{-i}|\theta_{i},\textbf{y}) ] d(\theta_{-i})
\bigg)
\bigg\}.
\end{align}
\end{itemize}
\end{corollary}

\begin{proof}
\textbf{\emph{(a)}} To start with, for each $i$ ($i=1,\cdots,K$), define a functional $\mathcal{F}_{-i}: \mathcal{Q}_{\theta_{-i}} \rightarrow \mathbb{R}$ that complements the functional $\mathcal{F}_{i}$ (\ref{eq:functional_theta_i}):
\begin{align}
\label{eq:functional_-theta_i}
&\mathcal{F}_{-i}\{q(\theta_{-i})\}
=
\mathbb{E}_{q(\theta_{-i})}[\log \pi(\theta_{i}|- \theta_{i}, \textbf{y}) ] - \text{KL}(q(\theta_{-i})||\pi(\theta_{-i}|\textbf{y}) ).
\end{align}
For each $i=1, \cdots, K$, use the duality formula (\ref{eq:duality_formula_continuous_version}) by replacing the $q(\theta)$, $p(\theta)$, and $h(\theta)$ in the formula with $q(\theta_{-i}) \in \mathcal{Q}_{\theta_{-i}}$, $\pi(\theta_{-i}|\textbf{y})\in \mathcal{Q}_{\theta_{-i}}$, and $\log \pi(\theta_{i}|\theta_{-i},\textbf{y})\in \mathcal{Q}_{\theta_{i}}$, respectively, which leads to
\begin{align}
\label{eq:CAVI_pf_1}
\log \pi(\theta_{i}|\textbf{y}) &= \text{sup}_{q(\theta_{-i}) \ll \pi(\theta_{-i}|\textbf{y})} \mathcal{F}_{-i}\{q(\theta_{-i}) \}.
\end{align}
Now, take the $\exp(\cdot)$ to the both sides of (\ref{eq:CAVI_pf_1}), and then change the $\exp(\cdot)$ and $\text{sup}(\cdot)$ to obtain
\begin{align}
\nonumber
\pi(\theta_{i}|\textbf{y}) &= \exp\ 
[
\text{sup}_{q(\theta_{-i}) \ll \pi(\theta_{-i}|\textbf{y})} \mathcal{F}_{-i}\{q(\theta_{-i}) \}
]\\
\nonumber
&=
\text{sup}_{q(\theta_{-i}) \ll \pi(\theta_{-i}|\textbf{y})} 
[
\exp\ 
\mathcal{F}_{-i}\{q(\theta_{-i}) \}
]
\\
\nonumber
&=\text{sup}_{q(\theta_{-i}) \ll \pi(\theta_{-i}|\textbf{y})} 
\bigg[
\frac{\exp \mathbb{E}_{q(\theta_{-i})}[\log \pi(\theta_{i}|\theta_{-i}, \textbf{y}) ]}{\exp \text{KL}(q(\theta_{-i})||\pi(\theta_{-i}|\textbf{y}))}
\bigg]
\\
\label{eq:CAVI_pf_2}
&\geq 
\text{sup}_{
q(\theta_{-i}) \ll \pi(\theta_{-i}|\textbf{y}),q(\theta_{-i}) \in \mathcal{Q}_{\theta_{-i}}^{MF}}
\bigg[
\frac{\exp \mathbb{E}_{q(\theta_{-i})}[\log \pi(\theta_{i}|\theta_{-i}, \textbf{y}) ]}{\exp \text{KL}(q(\theta_{-i})||\pi(\theta_{-i}|\textbf{y}))}
\bigg].
\end{align}
The inequality (\ref{eq:CAVI_pf_2}) holds because of a general property of supremum (i.e., it holds $\sup_{A}(\cdot) \geq \sup_{B} (\cdot)$ if $ B \subset A$). On the other hand, the CAVI-optimized variational density for $\theta_{-i}$, denoted as $q^{*}(\theta_{-i})$, can be represented by
\begin{align}
\label{eq:CAVI_pf_5}
q^{*}(\theta_{-i})=
\prod_{j=1, j\neq i }^{K} q^{*}(\theta_{j})
=
q^{*}(\theta_{1})\cdots q^{*}(\theta_{i-1})\cdot q^{*}(\theta_{i+1}) \cdots q^{*}(\theta_{K}) \in \mathcal{Q}_{\theta_{-i}|\textbf{y}}^{MF},
\end{align}
where each variational factor on the right-hand side has been optimized through the CAVI optimization formula (\ref{eq:variational_factor_theta_i}). Clearly, the density $q^{*}(\theta_{-i})$ (\ref{eq:CAVI_pf_5}) belongs to the set 
\begin{align*}
&B \defeq \{q: \Theta_{-i} \rightarrow [0,\infty) \,|\, q \text{ is a density supported on } \Theta_{-i}
,\,q(\theta_{-i}) \ll \pi(\theta_{-i}|\textbf{y}),\,q(\theta_{i}) \in \mathcal{Q}_{\theta_{-i}}^{MF} \}
\end{align*}
which is the set considered in the $\text{sup}(\cdot)$ (\ref{eq:CAVI_pf_2}). 

Now, use the definition of supremum and a simple calculation $a \times (1/a) =1$ to derive the following inequality
\begin{align}
\nonumber
\pi(\theta_{i}|\textbf{y}) &\geq 
\frac{\exp \mathbb{E}_{q^{*}(\theta_{-i})}[\log \pi(\theta_{i}|\theta_{-i}, \textbf{y}) ]}{\exp \text{KL}(q^{*}(\theta_{-i})||\pi(\theta_{-i}|\textbf{y}))}
\\
\nonumber
&=
\frac{\int \exp \mathbb{E}_{q^{*}(\theta_{-i})}[\log \pi(\theta_{i}|\theta_{-i}, \textbf{y})] d \theta_{i}}{\exp \text{KL}(q^{*}(\theta_{-i})||\pi(\theta_{-i}|\textbf{y}))}
\times 
\frac{\exp \mathbb{E}_{q^{*}(\theta_{-i})}[\log \pi(\theta_{i}|\theta_{-i}, \textbf{y}) ]}{\int \exp \mathbb{E}_{q^{*}(\theta_{-i})}[\log \pi(\theta_{i}|\theta_{-i}, \textbf{y})] d \theta_{i}}
\\
\label{eq:CAVI_pf_3}
&=\mathcal{R}_{-i}\{q^{*}(\theta_{-i}) \}
\times 
q^{*}(\theta_{i})  \quad\text{on } \Theta_{i},
\end{align}
where the functional $\mathcal{R}_{-i}\{\cdot \}: \mathcal{Q}_{\theta_{-i}}^{MF} \rightarrow (0,\infty)$ is defined by (\ref{eq:squashing_functional}) and $
q^{*}(\theta_{i}) \in \mathcal{Q}_{\theta_{i}|\textbf{y}}^{m}$ (\ref{eq:CAVI_pf_5}).

Finally, because $\pi(\theta_{i}|\textbf{y})$ and $q^{*}(\theta_{i})$ are densities, by taking $\int \cdot d\theta_{i}$ on the both sides of (\ref{eq:CAVI_pf_3}), we can further obtain $0 < \mathcal{R}_{-i}\{q^{*}(\theta_{-i})\} \leq 1$. Note that the value $\mathcal{R}_{-i}\{q^{*}(\theta_{-i})\}$ is constant with respect to the $i$-variational factor $q^{*}(\theta_{i})$.
\\
\\
\noindent \textbf{\emph{(b)}} For each $i$ ($i=1,\cdots,K$), by using the same reasoning used in proving \textbf{\emph{(a)}}, we can derive the following inequality 
\begin{align}
\label{eq:CAVI_pf_4}
\pi(\theta_{-i}|\textbf{y}) &\geq 
\frac{\exp \mathbb{E}_{q^{*}(\theta_{i})}[\log \pi(\theta_{-i}|\theta_{i}, \textbf{y}) ]}{\exp \text{KL}(q^{*}(\theta_{i})||\pi(\theta_{i}|\textbf{y}))}
=\mathcal{R}_{i}\{q^{*}(\theta_{i}) \}
\cdot
q^{*}(\theta_{-i})  \quad\text{on } \Theta_{-i},
\end{align}
where the functional $\mathcal{R}_{i}\{\cdot\}:\mathcal{Q}_{\theta_{i}}^{m} \rightarrow (0,\infty)$ is defined by
\begin{align*}
\mathcal{R}_{i}\{q(\theta_{i}) \}
&=
\frac{\int \exp \mathbb{E}_{q(\theta_{i})}[\log \pi(\theta_{-i}|\theta_{i}, \textbf{y})] d (\theta_{-i})}{\exp \text{KL}(q(\theta_{i})||\pi(\theta_{i}|\textbf{y}))},
\end{align*}
and $q^{*}(\theta_{-i})$ on the right-hand side of the (\ref{eq:CAVI_pf_4}) is obtained by interchanging the $i$ with $-i$ from the formula (\ref{eq:variational_factor_theta_i}).

Because $\pi(\theta_{-i}|\textbf{y})$ and $q^{*}(\theta_{-i})$ are densities, by taking $\int \cdot d(\theta_{-i})$ on the both sides of the inequality (\ref{eq:CAVI_pf_4}), we have $0 < \mathcal{R}_{i}\{q^{*}(\theta_{i})\} \leq 1$. Conclude the proof by using the fact that the Kullback-Leibler divergence $\text{KL}(q^{*}(\theta_{i})||\pi(\theta_{i}|\textbf{y}))$ is non-negative.
\end{proof}

\begin{figure}[h]
\centering
\includegraphics[width=1\textwidth]{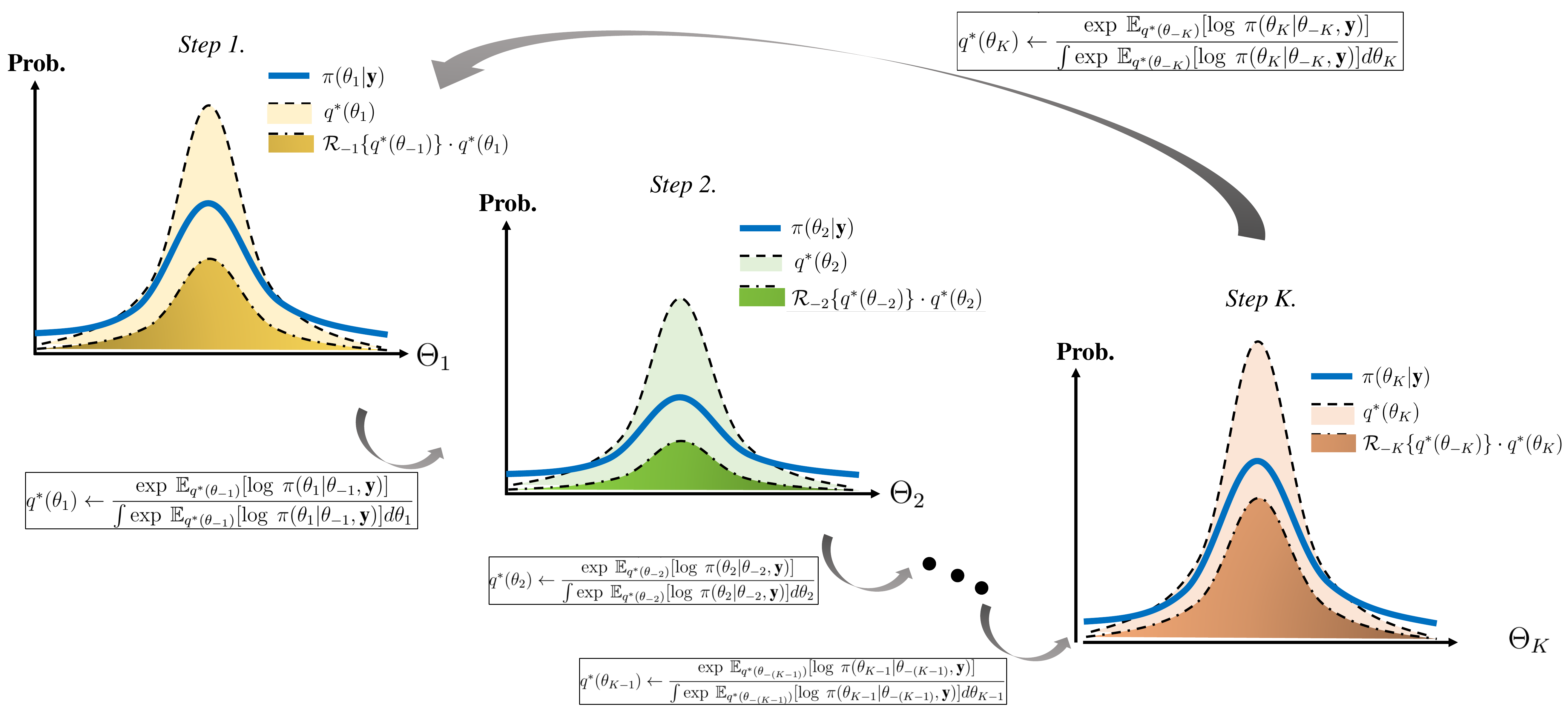}
\caption{Pictorial illustration of CAVI algorithm (Algorithm \ref{alg:CAVI algorithm}). At the $i$-th step within a cycle ($i=1,\cdots,K$), the panel shows that the $i$-th variational factor $q^{*}(\theta_{i})$ (\ref{eq:variational_factor_theta_i}) is pressed from above by the squashing constant $\mathcal{R}_{-i}\{ q^{*}(\theta_{-i})\} \in (0,1]$ so that it holds $\mathcal{R}_{-i}\{ q^{*}(\theta_{-i}) \} \cdot q^{*}(\theta_{i}) \leq \pi(\theta_{i}|\textbf{y})$ on $\Theta_{i}$. The distributional gap between $q^{*}(\theta_{i})$ and $\pi(\theta_{i}|\textbf{y})$ is explained by the inequality (\ref{eq:KL_lower_bound}).}
%In the bimodal situation, the squashing constant $\mathcal{R}_{-i}\{ q^{*}(\theta_{-i})\}$ may be smaller than the unimodal case. 
%In the bimodal situation, the squashing constant $\mathcal{R}_{-i}\{ q^{*}(\theta_{-i})\}$ will be small, which consequently leads to a higher $\text{KL}(q^{*}(\theta_{i})||\pi(\theta_{i}|\textbf{y}))$.
\label{fig:Cor_illustration}
\end{figure}
See Figure \ref{fig:Cor_illustration} for a pictorial illustration of the Corollary \ref{cor:CAVI_algorithm_revisited}. Corollary \ref{cor:CAVI_algorithm_revisited} \textbf{(a)} implies that, at the $i$-th step of the CAVI algorithm (Algorithm \ref{alg:CAVI algorithm}) ($i=1,\cdots,K$), the $i$-th variational factor $q^{*}(\theta_{i}) \in \mathcal{Q}_{\theta_{i}|\textbf{y}}^{m}$ (\ref{eq:variational_factor_theta_i}) and the $i$-th marginal target density $\pi(\theta_{i}|\textbf{y}) \in \mathcal{Q}_{\theta_{i}|\textbf{y}}^{m} $ are related by the inequality (\ref{eq:squashing_inequality}) with the squashing constant $\mathcal{R}_{-i}\{ q^{*}(\theta_{-i}) \}$, the functional value of $\mathcal{R}_{-i}\{\cdot\}$ evaluated at $q^{*}(\theta_{-i}) = \prod_{j = 1, j \neq i}^{K}q^{*}(\theta_{j}) \in \mathcal{Q}_{\theta_{-i}|\textbf{y}}^{MF}$. More colloquially, the variational factor $q^{*}(\theta_{i})$ (\ref{eq:variational_factor_theta_i}) can be pressed from above by the squashing constant $\mathcal{R}_{-i}\{ q^{*}(\theta_{-i})\}$, and kept below than $\pi(\theta_{i}|\textbf{y})$. Corollary \ref{cor:CAVI_algorithm_revisited} \textbf{(b)} suggests that the denominator in the CAVI formula (\ref{eq:variational_factor_theta_i}) plays an important role by participating as a lower bound of the distance $\text{KL}(q^{*}(\theta_{i})||\pi(\theta_{i}|\textbf{y}))$. Note that the lower bound is algorithm-based, which can be approximated via a Monte Carlo algorithm. 
\section{Summary}\label{sec:Summary}
In this paper, we aimed to provide some pedagogical insights about the Gibbs sampler (Algorithm \ref{alg:Gibbs sampler}) and CAVI algorithm (Algorithm \ref{alg:CAVI algorithm}) by treating fundamental densities used in the schemes as elements of sets of densities. Proofs of theorems contained in the paper are derived from the duality formula (\ref{eq:duality_formula_continuous_version}). The derived theorems helped comprehend some common structures between the Gibbs sampler and CAVI algorithm in a set-theoretical perspective. Among salient findings, one of the key discoveries was that the full conditional posterior distribution can be viewed as the global maximum of a functional induced by the duality formula. This has been extended to a new view on the Gibbs sampler from the perspective of information theory. Additionally, we showed that there is a link between the approximation quality of the CAVI algorithm and the denominator of the variational factor (\ref{eq:variational_factor_theta_i}). 

%\subsection*{\textbf{Funding}} 
%None. 
%\subsection*{\textbf{Availability of data and materials}} 
%None. 
%\subsection*{\textbf{Conflict of interests}} 
%The authors declare that they have no conflict of interest.
\bibliographystyle{chicago}
\bibliography{MFVB_refs}
\end{document}